\documentclass[12pt]{article}

\usepackage[numbered]{bookmark}
\usepackage{changepage}
\usepackage[utf8]{inputenc}
\usepackage{graphics, graphicx}
\usepackage{amsmath, amsfonts, amsthm, amssymb}
\usepackage{mathrsfs, mathtools, mathabx}
\usepackage{verbatim}
\usepackage{enumerate}
\usepackage{xcolor}
\usepackage{cite}
\usepackage{array}

\newcommand{\n}{\mathbb{N}}
\newcommand{\z}{\mathbb{Z}}

\newcommand{\re}{\mathbb{R}}
\newcommand{\al}{\alpha}

\newcommand{\la}{\lambda}

\newcommand{\de}{\delta}

\newtheorem{theorem}{Theorem}
\newtheorem{lemma}[theorem]{Lemma}
\newtheorem{corollary}[theorem]{Corollary}
\newtheorem{proposition}[theorem]{Proposition}
\theoremstyle{remark}
\newtheorem{remark}[theorem]{Remark}
\newtheorem{example}{Example}

\begin{document}
	\sloppy
	\title{Joining properties of automorphisms disjoint with all ergodic systems}
\author{P. Berk, M. G\'orska, T. de la Rue}

\maketitle
	\abstract{We study the class $Erg^\perp$ of automorphisms which are disjoint with all ergodic systems. We prove that the identities are the only multipliers of $Erg^\perp,$ that is, each automorphism whose every joining with an element of $Erg^{\perp}$ yields a system which is again an element of $Erg^{\perp}$, must be an identity. Despite this fact, we show that $Erg^\perp$ is closed {by} taking {Cartesian} products. Finally, we prove that there are non-identity elements in $Erg^\perp$ whose self-joinings {always} yield elements in $Erg^\perp$. This shows that there are non-trivial characteristic classes included in $Erg^\perp$.}
	
	\section{Introduction}
	The study of disjointness which expresses the maximal non-isomorphisms (in particular, the absence of non-trivial common factors) of two measure-preserving automorphisms originated in the seminal work of Furstenberg \cite{Fu} in 1967. It is an important direction of research till today and  determining whether two automorphisms are disjoint still remains a challenging problem. In particular, given a class $\mathcal A$ of automorphisms,  a full description of the class $\mathcal A^{\perp}$ of automorphisms disjoint with every element of $\mathcal A$ is often a hard task.
Let us recall some classical results:
\begin{enumerate}
\item[(i)] the class $ID^\perp$ of automorphisms disjoint with all identities equals the class $Erg$ of all ergodic transformations (by $ID$ we denote the class of all identities), see \cite{Fu};
		\item[(ii)] the class $ZE^\perp$ of maps disjoint with the zero entropy automorphisms is equal to the family of $K$-automorphisms ($\!\!$\cite{Fu}, together with \cite{Le-Pa-Th},\cite{Si}).
	\end{enumerate}
For some other classical classes only weaker relations are known:
the class $Dist$ of distal automorphisms is disjoint from the class $WM$ of weakly mixing automorphisms \cite{Fu} or
the class $Rig$ of rigid automorphisms is disjoint from the class $MM$ of  mildly mixing automorphisms \cite{Fu-We}. Usually, the problem to decide about the disjointness of two automorphisms slightly simplifies when we consider them both ergodic as it is reduced to study ergodic joinings. On the other hand, if two automorphisms are disjoint then one of them has to be ergodic (otherwise they both have non-trivial identities as factors and these are not disjoint).{ In connection with (i), it is natural to ask whether $Erg^\perp$ goes beyond identities}. It indeed does (this observation was a folklore), however, a satisfactory characterisation of elements of the class $Erg^\perp$ was only given  very recently in \cite{Au-Go-Le-Ru}, see Theorem~\ref{classification} below.  One of the recent interest to study the non-ergodic case (the elements of $Erg^\perp$, other than one-point system, are obviously non-ergodic) came recently from Frantzikinakis and Host  \cite{Fr-Ho}, who studied so called (logarithmic) Furstenberg systems  of the classical Liouville function and discovered that the celebrated Chowla conjecture may fail because some (hypothetic) of its Furstenberg systems might be elements of $Erg^\perp$.\footnote{{See also still open Problem 3.1 of workshop \cite{AIM} with Frantzikinakis' question whether an automorphism considered in Example~1 below can be realized as a Furstenberg system of the Liouville function.}}

Another natural problem when studying classes of the form $\mathcal A^\perp$, is to describe the class  $\mathcal M(\mathcal A^\perp)\subseteq\mathcal A^\perp$ of its multipliers, that is, of automorphisms $S\in\mathcal  A^\perp$ such that every joining of  $S$ with every element of $\mathcal A^\perp$ also belongs to $\mathcal A^\perp$. The study of this class is in general very difficult and leads often to many surprising results. {In 1989, Glasner and Weiss \cite{Gl-We} originated the study of the class $WM^\perp$, and they proved that $Dist\subsetneq WM^\perp$. Continuing on their result, in two papers \cite{Gl1}, \cite{Le-Pa}, it was proved that}
$$
Dist \subsetneq \mathcal{M}(WM^\perp)\subsetneq WM^\perp.$$
Returning to the $Erg$ class, we clearly have
\begin{equation}\label{zaleznosci}
ID\subseteq \mathcal{M}(Erg^\perp)\subseteq Erg^\perp,\end{equation}
where it is easy to see that $\mathcal{M}(Erg^\perp)\subsetneq Erg^\perp$.
In fact, in Section~\ref{sec: soloid} (see Example \ref{example1}) we consider a standard twist on the torus  $(x,y)\mapsto (x,x+y)$ as an example of an element in $Erg^\perp$ and  {directly} show that it is not an element of $\mathcal M(Erg^\perp)$. The latter {assertion follows also from our main, somewhat surprising, result}.\footnote{It was formulated as a conjecture by M.\ Lema\'nczyk in 2018 (private communication).}

\begin{theorem}\label{t:main}
We have $ID=\mathcal{M}(Erg^\perp)$.\end{theorem}

While the class $Erg^\perp$ is not closed under taking joinings, in Section~\ref{sec: products}, we prove that it is closed under Cartesian products.

{ To show that an ergodic automorphism $T$ is a multiplier of a class $\mathcal{A}^\perp$, it is  enough to show that the automorphisms determined by all self-joinings of $T$ are disjoint from the elements of $\mathcal{A}$, see for example Section~5 in \cite{Le-Pa} (see also Proposition \ref{prop: product})}. We show in Section \ref{sec: selfj} however that this approach fails when we study the class $Erg^\perp$ by exhibiting a non-identity automorphism $T$ whose all self-joinings yield elements of $Erg^\perp$ and which cannot be in $\mathcal{M}(Erg^\perp)$ by Theorem~\ref{t:main}. Moreover, the example constructed in Section \ref{sec: selfj}, serves to create a non-trivial \emph{characteristic class} (see Subsection \ref{subsec:charclass}), which does not contain an ergodic automorphism, yet it is not formed only of identities. This answers a question posed by Adam Kanigowski and Mariusz Lemańczyk in a private correspondence.

	
	\section{Preliminaries}
	\subsection{Measure-preserving automorphism and ergodic 
	decomposition}\label{subs:notation}
	We will consider automorphisms $T$ of standard Borel probability spaces 
	$(X,\mathcal{B},\mu)$ (the quadruple 
	$(X,\mathcal{B},\mu,T)$ is called a dynamical system). Recall that the measure $\mu$ can be written as a sum of its discrete and continuous part. We recall that, without loss of generality, we can assume that $X$ is a compact metric space and $T$ is a homeomorphism.
	Here $T:X\to X$ is an a.e. defined measurable bijection which is 
	$\mu$-preserving, {that is, $T_*\mu=\mu$, where $T_*(\cdot)$ denotes a push-forward of a measure}.
	Each Borel probability measure $\mu$ of a compact metric space $X$ yields a 
	standard Borel probability space. By $\mathscr{M}(X)$ we denote the space 
	of such measures (it is compact in the weak$^\ast$-topology). If $T:X\to X$ 
	is a homeomorphism  then by $\mathscr{M}(X,T)$ we denote the subspace of 
	$\mathscr{M}(X)$ of $T$-invariant measures (it is a nonempty closed subset 
	of $\mathscr{M}(X)$). 
	
	Let $(X,\mathcal{B},\mu)$ be a standard Borel probability space. By $\operatorname{Aut}(X,\mathcal{B},\mu)$ we denote the group of automorphisms of $(X,\mathcal{B},\mu)$. {It is a classical fact that $\operatorname{Aut}(X,\mathcal{B},\mu)$ with so called weak topology is
	 a Polish group (see ex. \cite{Kechris}).} Note that each $(X,\mathcal{B},\mu,T)$ defines a unitary operator $U_T$, called a \emph{Koopman operator}, of $L^2(X,\mathcal{B},\mu)$, $U_Tf=f\circ T$. We recall that $T$ is called \emph{ergodic} if the only $U_T$-invariant functions are constant. If additionally $U_T$ has no other eigenfunctions then $T$ is called \emph{weakly mixing}. 
	
	By $Aut$ we denote the space of all measure preserving transformations on all standard Borel probability spaces, considered up to isomorphism. By $Erg\subset Aut$, we denote the class of all ergodic automorphisms and by $WM\subset Erg$, the class of all weakly mixing automorphisms. Recall also that by $ID\subset Aut$, we denote the class of identities on all standard Borel probability spaces. Since we allow the measures under consideration to have atoms, the class $ID$ is uncountable.
	
	Of course, in general, $T\in \operatorname{Aut}(X,\mathcal{B},\mu)$ need not to be ergodic. Let us recall the classical concept of ergodic decomposition (see \cite{rokhlin}). Denote by $Inv(T)$ the $\sigma$-algebra of invariant sets. Let also
	\begin{equation}\label{eq: mudisint}
	\mu=\int_{X/Inv(T)}\mu_{\overline{x}}\,dP(\overline{x})
	\end{equation}
	be the disintegration of $\mu$ over $P:=\mu|_{Inv(T)}$. Then there exists a measurable partition of $\overline{X}:=X/Inv(T)$:
	\begin{equation}\label{def: ergcomp}
	\overline{X}=\bigcup_{n\geq1}\overline{X}_n\cup\overline{X}_\infty
	\end{equation}
	and a standard Borel probability space $(Z,\mathcal{D},\kappa)$, with $\kappa$ non-atomic, such that the space $(X,\mathcal{B},\mu)$ can be identified with
	the disjoint union of the corresponding product spaces (on $\mathbb Z_n:=\mathbb{Z}/n\mathbb{Z}$ we consider the uniform measure)
	$$
	\bigcup_{n\geq1}(\overline{X}_n\times\mathbb Z_n)\cup (\overline{X}_\infty\times Z)$$ and the action of $T$ in these new ``coordinates'' is given by
	\begin{equation}\label{ed}
		\begin{array}{c}
			T(\overline{x},i)=(\overline{x},T_{\overline{x}}(i))=(\overline{x},i+1)\text{ for }\overline{x}\in\overline{X}_n\text{ and}\\
			T(\overline{x},z)=(\overline{x},T_{\overline{x}}(z))\text{ for some ergodic }T_{\overline{x}}\in\operatorname{ Aut}(Z,\mathcal{D},\kappa), {\ x\in \overline X_\infty}.\end{array}\end{equation}
	The map $\overline{x}\mapsto T_{\overline{x}}$ is measurable in the 
	relevant Borel structures. The space $(\overline{X},P)$ is called the {\em 
	space of ergodic components}, and the representation of 
	$(X,\mathcal{B},\mu,T)$ given in \eqref{ed} is called the {\em ergodic 
	decomposition} of $T$. The ergodic decomposition is given up to a $P$-null 
	subset. For example, if $X=\mathbb{T}^2$ (considered with Lebesgue measure 
	$Leb_{\mathbb{T}^2}$) and $T:(x,y)\mapsto (x,y+x)$, then it is already the 
	ergodic decomposition of $T$ since $\mathbb{T}\times\{0\}$ is the space of 
	ergodic components  (with $P=Leb_{\mathbb{T}}$) and for $P$-a.e.\ $x\in 
	\mathbb{T}$, $T_x(y)=x+y$ is an (ergodic) irrational rotation.

	\subsection{Joinings}
	Let $(X,\mathcal B,\mu, T)$ and $(Y,\mathcal C,\nu,S)$ be two dynamical systems. We say that a  probability measure $\lambda$ on $(X\times Y, \mathcal B\otimes\mathcal C)$ is a \emph{joining }of $T$ and $S$ if:
	\begin{enumerate}
		\item $\lambda$ is $T\times S$- invariant;
		\item the marginals of $\lambda$ on $X$- and $Y$-coordinates are $\mu$ and $\nu$, respectively.
	\end{enumerate}
We note that $(X\times Y, \mathcal{B}\otimes \mathcal{C}, \lambda, T\times S)$ is a dynamical system (sometimes we denote such system simply by $T\vee S$).
We denote the set of all joinings of $T$ and $S$ by $J(T,S)$. 
Note that 
$\mu\otimes\nu$ is always an element of $J(T,S)$. Following \cite{Fu}, we say 
that $T$ and $S$ are \emph{disjoint} if $J(T,S)=\{\mu\otimes\nu\}$ and write 
$T\perp S$. 

In case when $T=S$, we set $J_2(T):=J(T,T)$ and refer to the elements of 
$J_2(T)$ as \emph{2-self-joinings}. Unless $T$ is the $1$-point dynamical 
system, it is never disjoint with itself. Indeed, the diagonal measure  {is a 
$2$-self-joining:\footnote{Here and thereafter we 
	denote $(f,g)(x)=(f(x),g(x))$.}} $\mu_{Id}:=(Id,Id)_*\mu$ and $\mu_{Id}\neq \mu \otimes \mu$. 
More generally, if $R\in C(T)$ is an element of the centralizer of $T$ (i.e. 
$R\in \operatorname{Aut}(X,\mathcal{B},\mu)$ and $R\circ T=T\circ R$), then the graph measure 
$\mu_R:=(Id,R)_*\mu$ is a member of $J_2(T)$. In particular, the 
\emph{off-diagonal} self-joinings $\mu_{T^n}$, $n\in \mathbb{Z}$, {belong to 
$J_2(T)$}. Besides, note that $C(T)$ is a closed subgroup of 
$\operatorname{Aut}(X,\mathcal{B},\mu)$.

It is a classical fact that in a weak-$\ast$ topology, $J_2(T)$ is a compact set,  see\cite{Gl2}. Moreover, if $T$ is additionally ergodic then $J_2(T)$ is a simplex and the set of extremal points of $J_2(T)$  consists of ergodic 2-self-joinings. In particular, the set of ergodic self-joinings is nonempty and we denote it by $J_2^e(T)$. Note that the graph joinings $\mu_R$ are always ergodic as long as $T$ is ergodic. In fact, the corresponding automorphisms are isomorphic to $T$, where an isomorphism is given by the map $x\mapsto (x,Rx)$.

Given $(X_i,\mathcal{B}_i,\mu_i, T_i) $, $i\geq 1$, we also consider infinite joinings $\lambda \in J(T_1,T_2,\ldots)$ (invariant mesures on $X_1\times X_2\times \ldots$ whose marginals are $\mu_i$, $i\geq 1$). If $T_i$ are ergodic then $J^e(T_1,T_2,\ldots)\neq \emptyset$. Note that if $A\subset \mathbb{N}$ then we can speak about $J(T_{i_1},T_{i_2}, \ldots)$, where $A=\{i_1,i_2,\ldots\}$ ($A$ can be finite here). Whenever $T_1=T_2=\ldots=T$, we speak about $J_{\infty}(T)$ the set of (infinite) self-joinings. Now, if $A\subset \mathbb{N}$ and $R_{i_k}\in C(T)$, $k\geq 2$, then we can consider the corresponding graph self-joining $\mu_{R_{i_2},R_{i_3},\ldots}:=(Id,R_{i_2}, R_{i_3},\ldots)_\ast \mu $. If $R_{i_k}$ are powers of $T$ then we speak about off-diagonal self-joinings. If each $\lambda\in J_{\infty}^e(T)$ is a  product of off-diagonal self-joinings then we say that the automorphism $T$ has the \emph{minimal self-joining property} (or \emph{MSJ}), see  \cite{Ru}.

 Following \cite{Ju-Ru}, we say that $(X,\mathcal B, \mu, T)$ has the \emph{PID property} if for every $\la\in J_{\infty}(X,\mathcal B,\mu, T)$, if $\la$ projects on every pair of coordinates as $\mu\otimes\mu$ then $\la$ is the product measure. Note that each $MSJ$ automorphism has the PID property.

Let $T_i:(X_i,\mathcal B_i,\mu_i)\to(X_i,\mathcal B_i,\mu_i)$ for $i=1,2$ be a pair of dynamical systems. Let also $\mathcal A_i\subset\mathcal B_i$  be factors of those systems, i.e. invariant sub-$\sigma$-algebras. Let $\mu_i =\int_{X/\mathcal A_i}\mu_{i,\bar x_i}\, d\mu_i|_{\mathcal A_i}$ be the disintegration of $\mu_i$ over the respective factor. Assume that $\lambda\in J(T_1|_{\mathcal A_1}, T_2|_{\mathcal A_2})$. Then the formula 
\[
\hat\lambda:=\int_{X_1/\mathcal A_1\times X_1/\mathcal A_2}\mu_{1,\bar x_1}\otimes\mu_{2,\bar x_2}\,d\lambda(\bar x_1,\bar x_2),
\]
 defines an element in $J(T_1,T_2)$ called \emph{the relatively independent extension of $\lambda$}. We proceed similarly for finite and countable families of automorphisms (cf. proof of Lemma \ref{lem: factorjoin}).

  We are interested in properties of the class of dynamical systems disjoint with all ergodic systems, i.e. with
 \[
 Erg^{\perp}:=\{T\in Aut;\, T\perp S\text{ for every }S\in Erg \}.
 \]
 In the following subsection we give some facts describing the structure of $Erg^{\perp}$ as well as some basic examples of elements from this class.
 
	\subsection{Elements of $Erg^{\perp}$}
	 Let us first recall some well-known examples of automorphisms from $Erg^\perp$. It is a classical fact that $Id\in Erg^{\perp}$. Also, every system of the form $(x,y)\mapsto (x,x+y)$ on $\mathbb T^2$ with invariant measure $\mu\otimes Leb_{\mathbb{T}}$ is an element of $Erg^{\perp}$, as long as $\mu$ is a continuous measure on $\mathbb{T}$ (see e.g. Theorem \ref{classification}). In fact, a recent result from \cite{Au-Go-Le-Ru} gives the full characterisation of elements of $Erg^\perp$ in the form of the following result.
	
	\begin{theorem}[$\!\!$\cite{Au-Go-Le-Ru}]\label{classification}
An automorphism $T$ belongs to $Erg^{\perp}$ if and only if ({see \eqref{eq: mudisint}}) for $P\otimes P$-almost every $(\bar x,\bar y)\in \overline{X}^2$ the automorphisms $T_{\bar x}$ and $T_{\bar y}$ are disjoint.
		\end{theorem}
Since the non-trivial elements of $Erg^{\perp}$ are never ergodic (none of the non-trivial system is disjoint with itself), the elements of the class $Erg^{\perp}$  in $\operatorname{Aut}(X,\mathcal{B},\mu)$ form a meagre set (as the ergodic systems form a generic subset of $\operatorname{Aut}(X,\mathcal{B},\mu)$ \cite{Ha}). Now, in view of Theorem \ref{classification} and the disjointness of time automorphisms result in \cite{Da-Ry}, the automorphism on $\mathbb T\times X$ given by $(t,x)\mapsto(t,T_t(x))$, where $(T_t)_{t\in\re}$ is a generic flow, is an element of $Erg^\perp$.

The following fact is also useful when trying to describe the class $Erg^{\perp}$.
	\begin{proposition}\label{ergdisj}
		Let $T \in Erg^\perp$ and let $R$ be an ergodic automorphism. Then $R$ is disjoint from $P$-a.e. fiber automorphism $T_{\bar x}$.
		\end{proposition}
	Indeed, in view of Theorem \ref{classification}, the measure $P$ is continuous. Then Proposition 
	\ref{ergdisj} follows from the following fact, proved in \cite{Au-Go-Le-Ru} 
	(note that the measure 
	$P|_{\overline{X}_\infty}$ has no atoms, { since otherwise $P\otimes P$ would have an atom in some point of the form $(\bar x, \bar x)$, with $T_{\bar x}$ not being a one-point system, which contradicts Theorem \ref{classification}}).
	\begin{proposition}\label{disjuptocount}
		Let $R$ be an ergodic automorphism. Then for every set $B\subset Aut$ of pairwise disjoint automorphisms, $R$ is disjoint with all but at most countably many elements of $B$.

	\end{proposition} 

	In Lemma \ref{lm: nonergeigenv} we prove that $\alpha \in \mathbb{S}^1$ is an eigenvalue of $T$ if 
	and only if $\alpha$ is an eigenvalue for a positive $P$-measure set of 
	fiber automorphisms $T_{\bar x}$. In 
	particular, if there exists $\al\in\mathbb S^1\setminus\{1\}$ such that 
	$\al$ is an eigenvalue for a set of positive $P$-measure fiber 
	automorphisms, then
	\begin{equation}\label{commonev}
	\begin{array}{c} \text{$T$ has the (ergodic) rotation by $\alpha$ as a 
	factor}\\\text{and $T$ is not an element of $Erg^{\perp}$.}\end{array}
	\end{equation}

 \subsection{Characteristic classes}\label{subsec:charclass}
Recall that a class $\mathcal F\subset Aut$ of measure preserving dynamical systems is \emph{characteristic} if it is closed under countable joinings and taking factors. In \cite{Ka-Ku-Le-Ru} the authors give a list of many examples of such classes. In particular, they proved that $ID$ - the class consisting of all identities of standard probability Borel spaces - is a characteristic class and is contained in every non-trivial characteristic class. A natural question arises, whether $ID$ is the only characteristic class that does not contain a non-trivial ergodic automorphism. {One of our goals is to answer positively to this question, by constructing such class inside $Erg^{\perp}$.}

First, we show that for every class $\mathcal A\subset Aut$, the set of all multipliers $\mathcal M(\mathcal A^{\perp})$ is a characteristic class. We prove the following general result.
\begin{lemma}\label{mult}
	Let $\mathcal A\subset {Aut}$. Then $(X,\mathcal B,\mu, T)\in 
	\mathcal A^{\perp}$ is  a multiplier of $\mathcal A^{\perp}$ if and only if 
	$(X^{\times \infty},\mathcal B^{\otimes \infty},\eta, T^{\times \infty})\in 
	\mathcal A^{\perp}$ is a multiplier of $\mathcal A^{\perp}$ for every $\eta\in 
	J_\infty(T)$.
\end{lemma}
\begin{proof}
{$\Longleftarrow$:} It is enough to notice that $T$ is isomorphic to the diagonal joining $(Id\times Id\times\ldots)_*\mu$.
	

	{$\Longrightarrow$:} {Assume now that $(X,\mathcal B,\mu,T)\in\mathcal A^\perp$ is a multiplier of $\mathcal A^\perp$. }Notice first that for every $n\ge 2$ and 
	every $\tilde\la\in J_n(T)$, the 
	automorphism $(X^{\times n},\mathcal B^{\otimes n},\tilde \la, T^{\times 
		n})$ is an element of $\mathcal A^{\perp}$ and is a multiplier of this 
	class. Indeed, since $T$ is a multiplier of $\mathcal A^{\perp}$ then for 
	$n=1$ and  $(Y,\mathcal{C},\nu, S)\in \mathcal{A}^\perp$ we have that  
	$T\vee S\in \mathcal A^{\perp}$. Again, using the fact that $T$ is a 
	multiplier of  $\mathcal A^{\perp}$, $T\vee (T\vee S)\in \mathcal 
	A^{\perp}$, {which settles the case $n=2$}. The argument follows by induction.
	
	Let now $(Y,\mathcal C,\nu, S)\in \mathcal A^{\perp}$ and let $(Z,\mathcal 
	D,\rho, R)\in \mathcal{A}$. Consider also $\eta\in J(T^{\times\infty})$ 
	and $\zeta\in J(T^{\times\infty}\vee S, R)$. Let 
	$\zeta_n:=\pi^{n,Y,Z}_*{\zeta}$ be the projection on the first $n$ 
	$X$-coordinates, $Y$-coordinate and $Z$-coordinate of $\zeta$. In 
	particular, $\zeta_n\in J(T^{\times n}\vee S,R)$. By the finite case, 
	$\zeta_n=\zeta_n|_{X^n\times Y}\otimes\rho$. Since finite cylinders generate the whole 
	product $\sigma$-algebra, by passing with $n$ to infinity, we obtain that 
	$\zeta=\eta\otimes \rho$.
\end{proof}
\begin{corollary}\label{cor:multclass}
	Let $\mathcal A\subseteq Aut$. Then $\mathcal M(\mathcal A^\perp)$ is a characteristic class.
\end{corollary}
\begin{proof}
In view of Lemma \ref{mult}, the class $\mathcal M(\mathcal A^\perp)$ is closed under taking countable joinings. It remains to see that it is closed under taking factors. Let $(X,\mathcal B,\mu, T)\in \mathcal M(\mathcal A^\perp)$ and let $(Y,\mathcal C,\nu, S)$ be a factor of $T$. If $S$ has a non-trivial joining with an ergodic system, then so does $T$, via the relatively independent extension. Hence $\mathcal M(\mathcal A^\perp)$ is also closed under taking factors.
\end{proof}

In view of the above corollary, a natural candidate for an example of a class that does not contain a non-trivial ergodic element and is larger than $ID$, would be the set $\mathcal M(Erg^\perp)$. However, in Section \ref{sec: soloid}, we show that this class contains only identities. To show that nonetheless such class exists, we present the following construction.

Let $T\in \operatorname{Aut}(X,\mathcal B,\mu)$. Let $\mathcal F(T)$ be the class of measure preserving dynamical systems, which consists of all countable self-joinings of $T$, as well as all factors of those joinings.

\begin{lemma}\label{lem: gencharclass}
	The class $\mathcal{F}(T)$ is characteristic.
\end{lemma}
This result follows from the fact that a factor of a factor of a fixed automorphism $R$ is still a factor of $R$ and from the following classical lemma, whose proof we provide below for the sake of completeness.
\begin{lemma}\label{lem: factorjoin}
	Let $\{(X_i,\mathcal B_i,\mu_i,T_i)\}_{i=1}^{\infty}$ be a family of measure preserving automorphisms. For every $i\in\n$, let $(Y_i,\mathcal C_i,\nu_i,S_i)$ be a factor of $T_i$ and let $F_i:X_i\to Y_i$ be the factorizing map. Then for any $\lambda\in J(S_1,S_2,\ldots)$, there exists a joining $\eta\in J(T_1,T_2,\ldots)$, such that $(S_1\times S_2\times \ldots,\lambda)$ is a factor of $(T_1\times T_2\times\ldots,\eta)$.
\end{lemma}
\begin{proof}
	For every $i\in\n$, consider the disintegration of $\mu_i$ with respect to the factor $S_i$:
	\[
	\mu_i=\int_{Y_i}\mu_{i,y_i}\, d\nu_i(y_i).
	\]
	Then consider the measure
	\[
	\eta:=\int_{Y_1\times Y_2\times\ldots} (\mu_{1,y_1}\otimes \mu_{2,y_2}\otimes \ldots)\,d\lambda(y_1,y_2,\ldots)
	\]
	on $X_1\times X_2\times\ldots$. It is a non-trivial element of $J(T_1,T_2,\ldots)$. Then $(S_1\times S_2\times\ldots, \lambda)$ is a factor of $(T_1\times T_2\times\ldots,\eta)$ through the following factorizing map
	\[
	(x_1,x_2,\ldots)\mapsto (F_1(x_1), F_2(x_2),\ldots).
	\]
\end{proof}

In Section \ref{sec: selfj} we provide an example of a non-identity automorphism $T\in Erg^\perp$ such that $\mathcal F(T)$ does not contain a non-trivial ergodic element, see Corollary \ref{cor: noergcharclass}.

%
%
%

\subsection{Group extensions}

Let $G$ be a compact abelian group with Haar measure $\lambda_G$ (for the tori groups $\mathbb{T}^n$ we will use notation $Leb_{\mathbb{T}^n}$) and let $(X,\mathcal B,\mu, T)$ be an ergodic system. Consider a \emph{cocycle} $\varphi:X\to G$, i.e.\ a measurable map. Then, let
\[
T_\varphi:X\times G\to X\times G;\ T_{\varphi}(x,g):=(Tx,\varphi(x)g)
\]
be the corresponding \emph{skew product}, it preserves the measure $\mu\otimes \lambda_G$. Moreover, $G$ acts on $X\times G$ via $(x,h)\mapsto (x,hg)$. This action also preserves the product measure and yields elements of $C(T_\varphi)$. We have the following criterion of ergodicity  for skew products.
\begin{theorem}[\!\!\cite{Fu}]\label{thm56}
	 Let $T$ be ergodic. Then the automorphism {$(X\times G, \mu\otimes\lambda_G,T_\varphi)$} is ergodic if and only if for any character $1 \neq \pi:G\to \mathbb{S}^1$ there is no measurable solution $F:X\to \mathbb{C}$, $|F|=1$ of the equation
	\begin{equation}\label{cohomeq}
	\pi(\varphi(x))F(x)=F(Tx).
	\end{equation}
	Moreover, assuming $T_\varphi$ is ergodic, $c$ is an eigenvalue of 
	$T_\varphi$ if and only if there exists a measurable solution 
	$F:X\to\mathbb C$, $|F|=1$ of the equation
		\[
	\pi(\varphi(x))F(x)=cF(Tx).
	\]

	\end{theorem}
We also need the following ``relative unique ergodicity'' type property of ergodic cocycles (i.e.\ of cocycles  for which $(T_\varphi,\mu\otimes\lambda_G)$ is ergodic).
\begin{theorem}\label{thm57}
	Assume that $\varphi:X\to G$ is an ergodic cocycle (over  {an ergodic} $T$) then $\mu\otimes\lambda_G$ is the only $T_{\varphi}$-invariant measure which projects on $\mu$.
\end{theorem}


\subsection{Elements of spectral theory}
Let us recall here some notions and facts concerning the spectral theory of dynamical systems. Let $(X,\mathcal B,\mu,T)$ be a dynamical system. For any $f\in L^2(X,\mathcal B,\mu)$ we define the \emph{spectral measure} $\sigma_{f,\mu}$ {(a finite positive measure defined on $\mathbb{S}^1$)} which, via the Herglotz theorem, is given by its Fourier transform:
\begin{equation}\label{spme}
\hat\sigma_{f,\mu}(n):= \int_{\mathbb{S}^1}z^n\,d\sigma_{f,\mu}=\langle f\circ T^{-n},f\rangle_{L^2(\mu)}:=\int_{X}f\circ T^{-n}\cdot \overline f\,d\mu
\end{equation}
for $n\in\mathbb Z$.
{Among spectral measures of functions in $$L^2_0(X,\mathcal B,\mu)=\left\{g\in L^2(X,\mathcal B,\mu):\  \int g\,d\mu=0\right\}$$ there are dominating ones (in the absolute continuity sense). All such must be equivalent, and their equivalence class is called the \emph{maximal spectral type} of $U_T$.}
If $\mu$ is understood, we write $\hat\sigma_{f}$ instead of $\hat\sigma_{f,\mu}$. Note that $\sigma_f(\mathbb{S}^1)=\|f\|_{L^2(\mu)}^2$.
Moreover, $f$ is an eigenfunction of $U_T$ if and only if $\sigma_{f,\mu}$ is the Dirac measure at the corresponding eigenvalue.

We now show some measurability results concerning spectral measures.

\begin{lemma}\label{msrbl:mes}
	Let  $X$ be a compact metric space and $T:X\to X$ be a homeomorphism. Let 
	$f\in C(X)$ be a complex-valued continuous function such that $|f|=1$. Then 
	the map $F:\mathscr M(X,T)\mapsto\mathscr M(\mathbb{S}^1)$ defined as
	\[
	F(\eta):=\sigma_{f,\eta}
	\]
	is continuous.
\end{lemma}
\begin{proof}
	Recall that by the definition of weak$^\ast$-convergence in 
	$\mathscr{M}(X)$, for every $g\in C(X)$ the map 
	$\eta\mapsto\int_{X}g\,d\eta$ is continuous. Thus the function
	\begin{equation}\label{msrbl1}
	\mathscr M(X,T)\ni\eta\xmapsto{{F_1}} (\hat\sigma_{f,\eta}(m))_{m\in 
	\n}=\left(\int_{X}f\circ T^{-m}\cdot \overline f\, d\eta\right)_{m\in 
	\mathbb Z}
	\end{equation}
	is continuous.
	It follows that the map $F_1:\mathscr M(X,T)\to \mathbb{D}^{\mathbb Z}$ 
	(recall that $\mathbb{D}=\{z\in \mathbb C:|z|\leq 1\}$ and on 
	$\mathbb{D}^\mathbb{Z}$ we consider the usual product metric $d$) given by
	\[
	F_1(\eta):= \left(\int_{X}f\circ T^{-m}\cdot \overline f\, d\eta\right)_{m\in \mathbb Z}
	\]
	is continuous. Since $\mathscr M(X,T)$ is compact, the image 
	$\Upsilon:=F_1(\mathscr M(X,T))$ is also compact in $\mathbb D^{\mathbb Z}$.
	
	Now, we define a function $F_2:\Upsilon\to \mathscr M(\mathbb S ^1)$ in the 
	following way:
	\[
	F_2((a_m)_{m\in\mathbb Z})=\sigma,\ \text{where}\ \hat\sigma(m)=a_m.
	\]
	It is well defined via the Herglotz theorem.
	Note that since the family $\{z^m\}_{m\in\mathbb Z}$ is linearly dense in $C(\mathbb S^1)$, the map $F_2$ is also continuous. Indeed, let $(a_m^n)_{m\in\mathbb Z},(b_m)_{m\in\mathbb Z}\in\Upsilon$ and $\sigma_n:=F_2((a_m^n)_{m\in\mathbb Z})$,
	$\sigma:=F_2((b_m)_{m\in\mathbb Z})$ and assume that $d((a_m^n)_{m\in\mathbb Z},(b_m)_{m\in\mathbb Z})\to 0$ as $n\to\infty$ in $\Upsilon$.
	Then, for every $m\in\mathbb Z$, we have $\int_{\mathbb T}z^m\,d\sigma_n\to \int_{\mathbb T}z^m\, d\sigma$. Since the functions $z^m$ are linearly dense, this yields the continuity of $F_2$. Since $F=F_2\circ F_1$, this finishes the proof.
\end{proof}
\begin{lemma}\label{msrbl:val}
	Let  $X$ be a compact metric space. Let $T:X\to X$ be a homeomorphism and 
	let $f\in C(X)$, $|f|=1$. Then, for every $\al\in\mathbb S ^1$, the map 
	$G:\mathscr M(X,T)\mapsto[0,1]$ defined as
	\[
	G(\eta):=\sigma_{f,\eta}(\{\al\})
	\]
	is measurable.
\end{lemma}
\begin{proof}
	Fix $\al\in\mathbb S ^1$.
	In view of Lemma \ref{msrbl:mes}, the map $F:\mathscr M(X,T)\mapsto\mathscr 
	M(\mathbb S^1)$ given by the formula $F(\eta)=\sigma_{f,\eta}$ is 
	continuous. In particular, the set $\Omega=F(\mathscr M(X,T))$ is compact 
	in the weak$^\ast$-topology. It is thus enough to show that the map 
	$G':\mathscr M (\mathbb S^1)\mapsto[0,1]$ defined as 
	$G'(\sigma):=\sigma(\{\al\})$ is measurable.
	
	Let $(g_n)_{n\in\n}\subset C(\mathbb S^1)$ be a  sequence of (bounded by 
	$1$) real continuous functions converging pointwise to the indicator 
	function $\chi_{\{\al \}}$. Define, for every $n\in\n$, the map 
	$G_n':\mathscr M(\mathbb S ^1)\to \mathbb{R}$ as 
	$G_n'(\sigma)=\int_{\mathbb T}g_n\,d\sigma$ which is continuous. Then, 
	$\lim_{n\to\infty}G_n'=G'$ pointwise. As a point limit of continuous 
	functions, $G'$ is measurable. Since $G=G'\circ F$, this finishes the proof.
\end{proof}
{
\begin{remark}\label{rem:zeroint}
Notice that if $F$ and $G$ are defined as $F(\eta):=\sigma_{f-\int f\, d \eta,\eta}$ and $G(\eta):=\sigma_{f-\int f\, d\eta,\eta}(\{\alpha\})$, then, by repeating proofs of Lemma \ref{msrbl:mes} and Lemma \ref{msrbl:val}, we get that $F$ is cotninuous and $G$ is measurable.
\end{remark}}
	{We will also make use of the following fact based on spectral theory that provides a tool to detect eigenvalues of non-ergodic dynamical systems.
	\begin{lemma}\label{lm: nonergeigenv}
	Let $X$ be a compact metric space. Let $\mu\in \mathscr M(X,T)$, where $T$ 
	is a homeomorphism of $T$. Let also 
	\[
	\mu=\int_{\overline X}\mu_{\bar x}\,dP(\bar x)
	\]
	be the ergodic decomposition of $T$ and let $T_{\bar x}$ denote the fiber automorphism, corresponding to $\bar x\in\overline X$. Then $\alpha\in\mathbb S^1$ is an eigenvalue of $(T,\mu_{\bar x})$ for a $P$-positive measure set of $\bar x$ if and only if, $\alpha$ is an eigenvalue of $(T,\mu)$.
	\end{lemma}
\begin{proof}
	Let $\alpha\in\mathbb{S}^1$ be as in the assumption. 
	Consider a sequence $\{f_n\}_{n\in \n}$ of functions which are dense in the space of complex continuous functions on $X$.
	Note that $\sigma_{f_n,\mu}=\int_{\overline{X}}\sigma_{f_n,\mu_{\bar x}}\,dP(\bar x)$. Indeed, the $m$-th coefficient of Fourier transform of the RHS measure equals
	\[
	\begin{split}
	\int_{\overline{X}}&\int_{\mathbb S^1} z^{-m}\, d\sigma_{f_n,\mu_{\bar x}}\,dP(\bar x)\\&=\int_{\overline{X}}\int_{X}f_n\circ T^{-m}\cdot \bar f_n(x)\,d\mu_{\bar x}(x)\,dP(\bar x)\\
	&=\int_{X}f_n\circ T^{-m}\cdot \bar f_n(x)\,d\mu(x)=\hat\sigma_{f_n,\mu}(m).
	\end{split}
	\]
	Then the following conditions are equivalent:
	\begin{itemize}
		\item $\al$ is an eigenvalue of $(T,\mu)$;
		\item there exists $n\in\n$ such that $\sigma_{f_n,\mu}(\{\alpha\})>0$;
		\item there exists $n\in\n$ and a $P$-positive measure set of $\bar x$ such that $\sigma_{f_n,\mu_{\bar x}}(\{\alpha\})>0$;
		\item there exists a $P$-positive measure set of $\bar x$ such that $\alpha$ is an eigenvalue.
	\end{itemize}

%
	\end{proof}
	}
	
\begin{corollary}\label{cor: Xn0mes}
Let $T\in Erg^\perp$. For any $\alpha\in\mathbb S^1\setminus\{1\}$, the set of ergodic components of $T$ which have $\alpha$ as an eigenvalue has $0$ measure wrt. $P$. In particular, $P(\overline X_n)=0$ for every $n\ge 2$.
\end{corollary}

The following fact is a folklore, but for the sake of completeness of the presentation we recall its proof.
\begin{lemma}\label{lem:wm}
	Let $X$ be a compact metric space. Let $\mu\in \mathscr M(X,T)$, where $T$ 
	is a homeomorphism of $T$. Let also 
	\[
	\mu=\int_{\overline X}\mu_{\bar x}\,dP(\bar x)
	\]
	be the ergodic decomposition of $T$ and let $T_{\bar x}$ denote the fiber automorphism, corresponding to $\bar x\in\overline X$. Then the set
	\[
	\mathcal{WM} :=\{\bar x{ \,\in \overline X };\   T_{\bar x}\text{ is weakly mixing} \}
	\]
	is measurable.
\end{lemma}
\begin{proof}
	{Let $\{f_n\}_{n\in\n}$ be a dense family in $C(X)$ and fix $f=f_n$}. 
	Recall that an automorphism is weakly mixing iff it has no non-trivial 
	eigenvalues or, in other words, all spectral measures on $L_0^2$ are 
	continuous. Let us recall that for any measure $\sigma\in\mathscr M(\mathbb 
	S^1)$, by Wiener's Lemma, we have
	\[
	\lim_{N\to\infty}\frac{1}{N}\sum_{i=0}^{N-1}|\hat\sigma(i)|^2={\sum_{\alpha \text{ is an atom of }\sigma} \sigma^2(\{\alpha\})}.
	\]
	{It is enough to show that above equality holds for $\sigma=\sigma_{f-\int 
	f\, d \mu,\mu}$}. Note that the function $H_N:\mathscr M(\mathbb S^1)\to 
	\mathbb R_{\ge 0}$ given by 
	$H_N(\sigma):=\frac{1}{N}\sum_{i=0}^{N-1}|\hat\sigma(i)|^2$, by the 
	definition of weak*-convergence,  is continuous. Thus the function 
	$H:\mathscr M(\mathbb S^1)\to \mathbb R_{\ge 0}$ given by 
	$H(\sigma):=\lim_{N\to\infty}H_N(\sigma)$, as a point-wise limit of 
	continuous functions, is measurable.

	By Remark \ref{rem:zeroint}, the map $F:\mathscr M(X,T)\mapsto \mathscr 
	M(\mathbb S^1)$ is continuous. Moreover, by the properties of 
	disintegration of measures, the assignment $E:\overline X\to \mathscr 
	M(X,T)$ given by $E(\bar x)=\mu_{\bar x}$ is also measurable. Thus the map 
	$H\circ F\circ E$ is measurable. It remains to notice that, {since $f$ is 
	arbitrary}, $\mathcal{WM}=(H\circ F\circ E)^{-1}(\{0\}).$
\end{proof}

We also recall that if $T\in \operatorname{Aut}(X,\mathcal{B},\mu)$ and $f,g\in L^2_0(X,\mathcal{B},\mu)$ then $f\perp g$ whenever $\sigma_f\perp \sigma_g$, i.e.\ when the spectral measures are mutually singular. It follows that two automorphisms are disjoint whenever the maximal spectral types on the relevant $L^2_0$ spaces are mutually singular \cite{Ha-Pa}.

	\section{Identity is a multiplier of $Erg^\perp$}\label{sec: idinM}
	
The following fact is classical {(and follows from the spectral disjointness of ergodic automorphism with the identity maps), we recall the classical proof for completeness:
\begin{proposition}\label{idperp}
	Any identity map is disjoint with all ergodic systems.
\end{proposition}}
\begin{proof} 
Assume that $(Z,\mathcal{D},\rho,R)$ is ergodic and consider the identity on $(Y,\mathcal{C},\nu)$. Let $\eta\in J(R,Id)$. Take $h\in L^2(Z,\rho)$ of zero mean and let $g\in L^2(Y,\nu)$. By the von Neumann theorem (and ergodicity)
	$$
	\frac1N\sum_{n=0}^{N-1}h\circ R^n\to 0\text{ in } L^2(Z,\rho),
	$$
	so the same convergence takes place also in $L^2(Z\times Y,\eta)$. Since the strong convergence implies the weak convergence,
	$$
	\frac1N\sum_{n=0}^{N-1}\int g\otimes h\circ R^n\,d\eta\to0.$$
	But, for each $N\geq1$,  $\int \frac1N\sum_{n=0}^{N-1} (g\otimes h\circ (Id\times R)^n)\,d\eta=\int g\otimes h\,d\eta$, whence $\int g\otimes h\,d\eta=0$.
	\end{proof}
	
{We can slightly generalize this result by showing the following:
\begin{lemma}\label{l:l4} Assume that $(Z,\mathcal{D},\rho,R)$, $(Z',\mathcal{D}',\rho',R')$ are ergodic and $R\perp R'$. Consider the identity on $(Y,\mathcal{C},\nu)$. Then  $R\perp R'\times Id$ (the latter automorphism, a fortiori, considered with product measure).\end{lemma}
\begin{proof}
Let $\eta\in J(R,R',Id)$. By the disjointness of $R$ and $R'$, $\eta|_{Z\times Z'}=\rho\otimes\rho'$ and $(R\times R',\rho\otimes\rho')$ is ergodic because their group of eigenvalues must be disjoint (modulo the eigenvalue~1). The claim follows from Proposition \ref{idperp}.

%
\end{proof}}

	We show in this section that any identity is actually also a multiplier of $Erg^\perp$.
	\begin{proposition}\label{idperp+}
	Let $(X,\mathcal B,\mu,T)\in Erg^{\perp}$ and let $(Y,\mathcal C,\nu, Id)$ be an identity map. Consider $\la\in J(T,Id)$. Then
	\[
	(X\times Y, \mathcal B\otimes\mathcal C, \la, T\times Id)\in Erg^\perp.
	\]
	\end{proposition}
\begin{proof}
	Assume that $(Z,\mathcal{D},\rho,R)\in Erg$ and let $\eta\in J(T,Id,R)$.
	Note that $\eta|_{X\times Y}=\lambda$, $\eta|_{Y\times Z}=\nu\otimes\rho$ and (by assumption)
	\begin{equation}\label{eee1}
		\eta|_{X\times Z}=\mu\otimes \rho.\end{equation}
	Fix bounded, real functions $f\in L^2(X,\mu)$, $g\in L^2(Y,\nu)$ and $h\in L^2_0(Z,\rho)$. All we need to show is that
	\begin{equation}\label{eee0}\int f\otimes g\otimes h\,d\eta=0.\end{equation}
	Let $f=f_1+f_2$, where $f_1\circ T=f_1$ and $f_2\perp L^2(Inv(T))$. Then
	$$
	\int f\otimes g\otimes h\,d\eta=\int f_1\otimes g\otimes h\,d\eta+\int f_2\otimes g\otimes h\,d\eta.$$
	Now, the spectral measure (all the spectral measures are computed in $L^2(X\times Y\times Z,\eta)$) of the function $f_1\otimes g$ is equivalent to $\delta_{1}$, while the spectral measure of $h$ has no atom at 1 since $R$ is ergodic. Hence,  these spectral measures are mutually singular and therefore $f_1\otimes g$ and $h$ are orthogonal in $L^2(X\times Y\times Z,\eta)$, so the first term on the RHS disappears. In view of \eqref{eee1}, we have
	$$
	\sigma_{f_2\otimes h}=\sigma_{f_2}\ast \sigma_h.$$
	Suppose that this measure has an atom at~1. Then both spectral measures $\sigma_{f_2}$ and $\sigma_h$ must have an atom at $c\in \mathbb{S}^1$ and $\bar c\in \mathbb{S}^1$ respectively, where $c\neq1$ (as $R$ is ergodic). But then $T$ cannot be in $Erg^\perp$, a contradiction. Hence,  the spectral measures of $f_2\otimes h$ and $g$ are mutually singular hence these functions are orthogonal in $L^2(X\times Y\times Z,\eta)$ and \eqref{eee0} holds.
\end{proof}

{We now present an alternative proof showing that the result is a consequence of Theorem~\ref{classification}.}
\begin{proof}
	Let
	\[
	\mu=\int_{\overline{X}}\mu_{\bar x}\,dP(\bar x)
	\]
	be the ergodic decomposition of {$(T,\mu)$}, i.e. $P=\mu|_{Inv (T)}$ and $\overline{X}=X/Inv(T)$. Consider the disintegration of $\la$ with respect to $Inv (T)$:
	\begin{equation}\label{eq:ladeco}
	\la=\int_{\overline{X}}\la_{\bar x}\,dP(\bar x),
	\end{equation}
	where the equality of the integrating measures as well as the fact that
	\begin{equation}\label{eq:projident}
	\pi_*^X\la_{\bar x}=\mu_{\bar x}\quad \text{ for $P$-a.e. }\bar x\in \overline{X}
	\end{equation}
	follow from the uniqueness of disintegrations. Note also that the measures $\lambda_{\overline x}$ are $T\times Id$-invariant. Let $\nu_{\bar x}=\pi_*^Y\la_{\bar x}$. Since Proposition \ref{idperp} holds, by \eqref{eq:projident}, we get
	\begin{equation}\label{eq:ergdecoid}
	\la_{\bar x}=\mu_{\bar x}\otimes\nu_{\bar x}
	\end{equation}

	Let $(Z,\mathcal D,\rho, R)$ be an ergodic transformation.
	By Theorem \ref{classification} and Proposition \ref{ergdisj} we obtain that
	\begin{equation}\label{eq:aedisjoint}
	(R,\rho)\perp(T,\mu_{\bar x})\quad\text{ for $P$-a.e. } \bar x\in \overline{X}.
	\end{equation}
	Consider a joining
	\[
	\eta\in J\left((T\times Id,\la),(R,\rho)\right).
	\]
	Consider the disintegration of $\eta$ with respect to $\sigma$-algebra $Inv(T)\otimes \{\emptyset,Y\}\otimes \{\emptyset,Z\}$, that is
	\begin{equation}\label{eq:etadeco}
	\eta=\int_{\overline X}\eta_{\bar x}\,dP(\overline x),
	\end{equation}
	where again the appearance of  $P$ as integrating measure as well as the fact that
	\[
	\pi^{X\times Y}_*\eta_{\bar x}=\la_{\bar x}
	\]
	follows from the uniqueness of disintegration. Moreover, the measures $\eta_{\bar x}$ are $T \times Id\times R$-invariant and $\pi^{Z}_*\eta_{\bar x}=\rho$. {By applying Lemma~\ref{l:l4} to  \eqref{eq:ergdecoid} and \eqref{eq:aedisjoint}, we get that}
	\[
	\eta_{\bar x}=\la_{\bar{x}}\otimes\rho,
	\]
	thus, by \eqref{eq:etadeco} and \eqref{eq:ladeco}, we get
	that
	\[
	\eta=\la\otimes\rho.
	\]
	\end{proof}

\section{Identities are the only multipliers of $Erg^\perp$}\label{sec: soloid}
In this section we prove that there is no non-identity element of $Erg^\perp$ which is a multiplier of this class. First we show that the twist transformation on the torus is not a multiplier of $Erg^\perp$.


\begin{example}\label{example1}
Let $T:X\times \mathbb T\to X\times\mathbb T$ be given by $T(x,y)=(x,y+\beta(x))$, where $\beta:X\to \mathbb T$ is measurable. Assume that $\rho\in \mathscr{M}(X)$ is a measure on $X$ such that for every $z\in \mathbb T$ we have $\rho\{x\in X,\ \beta(x)=z\}=0$. In particular $\rho$ is continuous. Note that since for every $\alpha\in \mathbb T$ there is only counably many non-disjoint rotations, the 
assumptions of Theorem \ref{classification} are satisfied and, therefore $T\in 
Erg^\perp$.\footnote{Another, more elementary, proof follows by the fact that 
almost every ergodic component of $T$ is disjoint (even spectrally disjoint) 
with a fixed ergodic automorphism $R$ and, by the ergodicity of $R$, any 
joining between $R$ and $T$ is a convex combination of joinings between $R$ and 
ergodic components of $T$.} We will show 
that the system $(T,\rho\otimes Leb_{\mathbb{T}})$ is not a 
multiplier of  $Erg^{\perp}$.

	Let $\alpha \in \mathbb{R}\setminus \mathbb{Q}$ and consider the 
	automorphism
	\[
	R(x,z)=(x,z+\beta(x)+\alpha)
	\]
	 on $\mathbb{T}^2$, which preserves the measure $\rho\otimes Leb_{\mathbb 
	 T}$. One can check that $R$ also satisfies the assumptions 
	 of Theorem \ref{classification}, hence $R\in Erg^\perp$.
Now, consider the transformation $P$ on $X\times \mathbb{T}^2$ given by $P(x,y,z)=(x,y+\beta(x),z+\beta(x)+\alpha)$. It is easy to see that $P$ is an automorphism of  $(X\times \mathbb{T}^2,\rho\otimes Leb_{\mathbb{T}} \otimes Leb_{\mathbb{T}})$. Moreover, we can treat $\rho\otimes Leb_{\mathbb{T}}  \otimes Leb_{\mathbb{T}}$ as a measure on $X^2\times\mathbb{T}^2$  (up to a permutation of coordinates):
\[
\rho\otimes Leb_{\mathbb{T}} \otimes Leb_{\mathbb{T}} (A_1\times A_2\times A_3 \times A_4)=\rho(A_1\cap A_3) Leb_{\mathbb{T}}(A_2) Leb_{\mathbb{T}}(A_4),
\]
where $A_i\in \mathcal{B}(\mathbb{T})$ for $i=1,2,3,4$.
Notice that then we have
\[
\rho\otimes Leb_{\mathbb{T}}\otimes Leb_{\mathbb{T}}(\{(x,y,x,z): x\in X, \ y,z \in \mathbb{T}\})=1
\]
and the measure we consider is $T\times R$-invariant, so it is easy to see that it is a joining of automorphisms $T$ and $R$. Now, notice that $( T\times R,\rho\otimes Leb_{\mathbb{T}}\otimes Leb_{\mathbb{T}})$ has the rotation by $\alpha$ as a factor. Indeed, consider the map $\Pi:(X\times \mathbb{T}^2,\rho\otimes Leb_{\mathbb{T}} \otimes Leb_{\mathbb{T}})\to(\mathbb{T},Leb_{\mathbb{T}} )$ given by $\Pi(x,y,z)=z-y$. Then we have:
\[
\Pi(P(x,y,z))=\Pi(x,y+\beta(x),z+\beta(x)+\alpha)=z-y+\alpha
\]
and
\[
R_\alpha(\Pi(x,y,z))=R_\alpha(z-y)=z-y+\alpha,
\]
where $R_\alpha$ denotes the ergodic rotation by $\alpha$ on $\mathbb{T}$.
In particular, in view of \eqref{commonev}, the automorphism $(T\times R,\rho\otimes Leb_{\mathbb{T}}\otimes Leb_{\mathbb{T}})$ is not an element of $Erg^\perp$.
	
\end{example}

{
\begin{remark} Note that the automorphisms considered in Example~\ref{example1} 
are also interesting from the spectral theory of dynamical systems viewpoint. 
Indeed, it is not hard to see that the maximal spectral type of $T$ in Example 
\ref{example1} on $L^2_0(\mathbb T^2, \rho\otimes Leb_{\mathbb T})$ is of the 
form \begin{equation}\label{realiz}
\sum_{
	k\in\mathbb Z\setminus\{0\}}\frac1{2^{|k|}}(e^{2\pi ik\cdot})_\ast(\rho)\end{equation}
In particular, for each measure $\rho$ on $\mathbb T$, the measure \eqref{realiz} has a realization as the maximal spectral type of a Koopman operator.
\end{remark}}

We say that the map $R\in \operatorname{Aut}(X,\mathcal B,\mu)$ is \emph{compact} if the subgroup generated by $R$ is compact. 
The next lemma is a well-known result, see e.g. \cite{Ju-Ru}.  It allows to 
represent an arbitrary ergodic transformation with a compact element in its centralizer in the form of an  algebraic 
extension.
\begin{lemma}\label{lem:skewrep}
	Assume that $(X,\mathcal B,\mu, T)$ is ergodic and $R\in C(T)$ be compact. 
	Moreover, assume that $G:=\overline{\{R^n:\ n\in\z\}}$, equipped with the 
	Haar measure $\la_G$, acts freely on $X$. Then $T$ is isomorphic to the 
	extension $\bar T_{\varphi}$ over $\bar T$ which is the action of $T$ on 
	$X/ Inv(R)$ and $\varphi:X/ Inv(R)\to G$, where $\bar T_{\varphi}$ 
	preserves $\mu|_{Inv(R)}\otimes \lambda_G$. Moreover, the group $G$ acts 
	trivially on the first coordinate of the above extension and it is the 
	multiplication on the second coordinate. 
\end{lemma}
We have the following result, which says that under relevant assumptions, the property of two cocycles being disjoint is invariant under shifting the cocycle extensions by the same group element.
\begin{lemma}\label{lem:disher}
If $(T_\varphi,X\times G,\mu\otimes\la_G)$ and $(T_{\varphi'}',Y\times G,\nu\otimes\la_G)$ are weakly mixing and disjoint, then for every $g_0\in G$ the automorphisms $T_{\varphi(\cdot)g_0}$ and $T_{\varphi'(\cdot)g_0}'$ are also weakly mixing and disjoint.
\end{lemma}
\begin{proof}
	By Theorem \ref{thm56}, if $T_{\varphi(\cdot)g_0}$ is not weakly mixing, then there exists a measurable solution $F$ to the equation
	\[
	\pi(\varphi(x))\pi(g_0)F(x)=cF(Tx),
	\]
	which contradicts weak mixing of $T_{\varphi}$. Since $T_{\varphi}\perp T_{\varphi'}'$, then $T\perp T'$. Hence, if $\rho\in J(T_{\varphi(\cdot)g_0},T_{\varphi'(\cdot)g_0}') $ then $\rho|_{X\times X'}=\mu\otimes\mu'$. Now since $T_{\varphi(\cdot)g_0},T_{\varphi'(\cdot)g_0}'$ are weakly mixing, so is their product and the result follows from Theorem \ref{thm57}.

	\end{proof}

\begin{lemma}\label{RWM}
	Assume that $(X,\mathcal B,\mu, T)$ is weakly mixing, $R\in C(T)$ is compact and $G:=\overline{\{R^n;\ n\in\z\}}$ acts freely on $X$. Then $T\circ R$ is weakly mixing.
	\end{lemma}
\begin{proof}
	By Lemma \ref{lem:skewrep}, $T$ is isomorphic to the compact group 
	extension over the factor $Inv(R)$ with some cocycle $\varphi:X/Inv(R)\to 
	G$. If we identify $x$ with the pair $(\bar x,S_x)$, then 
	\[
	T(\bar x,S_x)=(\bar{T}\bar x, \varphi(\bar x)\circ S_x).
	\]
	We have 
	\[
	T\circ R(\bar x,S_x)=T(\bar x,S_x\circ R)=(\bar{T}\bar x, \varphi(\bar x)\circ S_x\circ R)=
	(\bar{T}\bar x, (\varphi(\bar x)\circ R)\circ S_x)
	\]
	Then  the result  follows from Lemma \ref{lem:disher}.
	\end{proof}

\begin{lemma}\label{lem:ergfac}
	Assume that $(X,\mathcal B,\mu, T)$ is weakly mixing and $R\in C(T)$ is compact, $R\neq Id$. Then there exists $\la\in J(T,T\circ R)$ which has the translation by $R$ on $G:=\overline{\{R^n;\,n\in\z \}}$ as a factor. In particular, it has a non-trivial ergodic factor.
\end{lemma}
\begin{proof}
	Recall that both $T$ and $T\circ R$ can be seen as extensions over $\bar T$ with cocycles $\varphi$ and $\varphi (\cdot)R$, respectively.
	It is enough then to consider the measure $\la=\left((Id\times Id)_*\mu\right)\otimes\la_G\otimes \la_G$, which corresponds to the automorphism $(\bar x,g,h)\mapsto (\bar Tx, \varphi(x)\circ g, \varphi(x)\circ R\circ h)$. Then the map $(\bar x,g,h)\mapsto h\circ g^{-1}$ yields the translation by $R$ on $G$ as a factor. Since it is a rotation by a generator on a compact group, it is ergodic. Indeed, choose a non-trivial character $\chi\in \widehat G$ s.t. $\chi(R)\neq 1$. Then use $\chi$ to define a rotation by $\chi(R)$ as a factor on $\chi(G)$.
	\end{proof}
As one of the crucial steps to prove the main result of this section, we show first its important special case .
%
%
\begin{lemma}\label{prop: wmcase}
	Let $T\in Erg^{\perp}$ be such that almost every ergodic component is weakly mixing and not a one-point system. Then $T$ is not a multiplier of $Erg^{\perp}$.
\end{lemma}
\begin{proof}
	Assume that $(X,\mathcal B,\mu, T)\in Erg^{\perp}$ is a multiplier of $Erg^{\perp}$ and let
	\[
	\mu=\int_{\overline{X}}\mu_{\bar x}\,dP(\bar x)
	\]
	be the ergodic decomposition of $(T,\mu)$. Let  $\tilde{\mu}\in J_2(T,\mu)$ be given by
	\begin{equation}\label{ergdecomp}
	\tilde\mu=\int_{\overline{X}}\mu_{\bar x}\otimes\mu_{\bar x}\,dP(\bar x).
	\end{equation}
	Note that it is the ergodic decomposition of $(T\times T, \tilde\mu)$ since almost every fiber automorphism $(T\times T,\mu_{\bar x}\otimes\mu_{\bar x})$, by assumption, is weakly mixing. Note also that $(T\times T,\tilde\mu)\in Erg^{\perp}$, in fact, it is even a multiplier of $Erg^{\perp}$ as a consequence of Lemma \ref{mult}.
	
	On $X\times X$ we consider the homeomorphism $R$ given by the formula $R(x,y)=(y,x)$. Note that $R$ preserves measures $\mu_{\bar x}\otimes\mu_{\bar x}$ for every $\bar x\in \overline{X}$, thus, it preserves $\tilde\mu$. Furthermore,
	\[
	R\in C(T\times T,\mu_{\bar x}\otimes\mu_{\bar x}) \text{ for every $\bar x\in \overline{X}$. }
	\]
	Note that $G:=\overline{\{R^n:\ n\in\mathbb{Z}\}}=\{Id, R\}$ is isomorphic to $\mathbb Z_2=\mathbb Z/2\mathbb Z$ and  $R$ acts on it by addition $i\mapsto i+1\ \operatorname{mod} 2$. In particular, the action of $R$ on $G$ is ergodic.
	
	Consider the equivalence relation $\sim$ on  $X\times X$, where the equivalence class of $(x,y)$ is $\{(x,y),(y,x)\}$. Let $\pi:X\times X\to X\times X/\sim$ be the natural quotient map and let $\psi:X\times X\to \mathbb Z_2$ be a measurable function such that $\psi(x,y)=\psi(y,x)+1$, whenever $x\neq y$. While the existence of such a measurable function follows from an appropriate selector theorem, we give a direct argument. Since $X\times X$ is a standard Borel space, there exists a measurable 1-1 function $\gamma:X\times X\setminus\{(x,x):\, x\in X\}\to \re$. Then, it is enough to put $\psi(x,y)=0$ whenever $\gamma(x,y)<\gamma(y,x)$ and $\psi(x,y)=1$ otherwise.
	
	Since $(T,\mu_{'\bar x})$ is not a one-point system for every $\bar x\in\overline X$, the group $G$ acts freely on $(X\times X, \mu_{\bar x}\otimes\mu_{\bar x})$. Hence, we can apply
	Lemma \ref{lem:skewrep} and get that  the 
	automorphism $(T\times T,\mu_{\bar x}\otimes\mu_{\bar x})$ is isomorphic, 
	via the map $(\pi,\psi)$ to the automorphism
	\begin{equation} \label{quot}\text{$\left((\overline{T\times T})_{\varphi},\pi_*(\mu_{\bar x}\otimes\mu_{\bar x})\otimes {\tfrac{\de_0+\de_1}{2}}\right)$ on $(X\times X/\sim)\times \mathbb Z_2$},
	\end{equation} where $\varphi$ is defined as  in Lemma \ref{lem:skewrep} 
	(in our case $\varphi=\psi\circ(T\times T)-\psi$). Note that $\pi$, $\psi$ 
	and $\varphi$ do not depend on $\bar x\in X\times X/\sim$. By the choice of 
	$\psi$ and the continuity of $\pi$ on $X\times X$ as well as on $\mathscr 
	M(X\times X)$, the assignment $\bar x\mapsto \pi_*(\mu_{\bar 
	x}\otimes\mu_{\bar x})\otimes{\tfrac{\de_0+\de_1}{2}}$ is measurable. 
	Hence,  we can consider simultaneously each fiber automorphism as the same 
	extension with varying invariant measure or, in other words, $(T\times 
	T,\tilde\mu)$ is isomorphic to the automorphism $(\overline{T\times 
	T})_{\varphi}$ considered with measure
	\[
	\int_{\overline{X}}\pi_*(\mu_{\bar x}\otimes\mu_{\bar x})\otimes {\tfrac{\de_0+\de_1}{2}}\,dP(\bar x).
	\]
	
Since each of the fiber measures $\mu_{\bar x}\otimes\mu_{\bar x} $ is 
$R$-invariant, in view of Lemma \ref{RWM}, the automorphism $\left((T\times 
T)\circ R,\mu_{\bar x}\otimes\mu_{\bar x} \right)$ is weakly mixing. Thus 
\eqref{ergdecomp} is also an ergodic decomposition of  $((T\times T)\circ 
R,\tilde\mu)$. Moreover, by  Lemma \ref{lem:disher} and Theorem 
\ref{classification} the automorphism $\left((T\times T)\circ 
R,\tilde\mu\right)$ is also an element of $Erg^{\perp}$. Finally, following 
\eqref{quot} and Lemma \ref{lem:skewrep}, we have that the automorphism 
$((T\times T)\circ R,\mu_{\bar x}\otimes\mu_{\bar x})$ is isomorphic to the 
automorphism
\begin{equation} \text{$\left((\overline{T\times T})_{\varphi(\cdot) R},\pi_*(\mu_{\bar x}\otimes\mu_{\bar x})\otimes {\tfrac{\de_0+\de_1}{2}}\right)$ on $(X\times X/\sim)\times \mathbb Z_2$}
\end{equation} via the map $\pi\times\psi$.

Consider now the map
\[\begin{split}
&\hat T:(X\times X/\sim)\times \mathbb{Z}_2\times\mathbb{Z}_2\to(X\times X/\sim)\times \mathbb{Z}_2\times\mathbb{Z}_2;\\
&\hat T(\xi,i,j):=((\overline{T\times T})\xi,i+\varphi(\xi),j+\varphi(\xi)+1).
\end{split}
\]
Note that $\hat T$ preserves the measure
\[
\hat\mu:=\int_{\overline{X}} \pi_*(\mu_{\bar x}\otimes\mu_{\bar x})\otimes \tfrac{\de_0+\de_1}{2}\otimes \tfrac{\de_0+\de_1}{2}\,dP(\bar x)
\]
and $(\hat T,\hat\mu)$ corresponds to a joining of $(T\times T,\tilde\mu)$ and $((T\times T)\circ R,\tilde\mu)$ of the form $\int_{\overline{X}}(Id\times Id)_*\pi_*(\mu_{\bar x}\otimes\mu_{\bar x})\otimes\tfrac{\de_0+\de_1}{2}\otimes \tfrac{\de_0+\de_1}{2}\,dP(\bar x)$. Since $((T\times T)\circ R,\tilde\mu)\in Erg^{\perp}$ and $(T\times T,\tilde\mu)$ is a multiplier of $Erg^\perp$, it follows that $(\hat T,\hat\mu)\in Erg^{\perp}$. However the map $(\xi,i,j)\mapsto j-i$ yields rotation by $1$ on $\mathbb Z_2$ as a factor of $(\hat T,\hat\mu)$ which is a contradiction. Hence,  there is no multiplier of $Erg^{\perp}$ with  non-trivial weakly mixing fibers.

%
	\end{proof}

Before passing to prove the main result of this section, let us deal with the measurability issues concerning the choice of a discrete factor. For this purpose, let us recall the classical Kuratowski-Ryll-Nardzewski Theorem.
\begin{theorem}[Kuratowski-Ryll-Nardzewski, \cite{KuRyN}]\label{th:selector}
Let $(X,\mathcal B)$ be a Polish space with the Borel $\sigma$-algebra and let 
$(\Omega,\mathcal C)$ be a measurable space. Let $f$ be a multifunction on 
$\Omega$ which takes values in closed subsets of $X$. Assume moreover that $f$ 
is weakly measurable, that is,  for every open $A\subset X$, the set 
$\{\omega\in\Omega;\,f(\omega)\cap A\neq\emptyset\}$ is measurable. Then, there 
exists a measurable selector $F:\Omega\to X$ of $f$, i.e.\  $F(\omega)\in 
f(\omega)$ for every $\omega\in\Omega$.
\end{theorem}
We use the above result to show that we can measurably assign an eigenvalue to 
any element of ergodic decomposition. Recall the notions of space of ergodic 
components $\overline{X}$ and its subsets $\overline{X}_n$ and 
$\overline{X}_{\infty}$ from \eqref{def: ergcomp}.
\begin{lemma}\label{prop:measfactor}
Let $T\in \operatorname{Aut} (X,\mathcal B, \mu)$  and let
\[
\mu=\int_{\overline{X}}\mu_{\bar x}\,dP(\bar x)
\]
be the ergodic decomposition. Assume that $P$-a.e. fiber automorphism is 
non-weakly mixing (i.e.\ possesses a non-trivial eigenvalue). Then there exists 
a measurable map $F:\overline{X}_{\infty}\to L^2(Z,\mathcal D, \kappa)$, where  
$(Z,\mathcal D, \kappa)$ is defined in \eqref{ed}, such that $F(\bar x)$ is an 
eigenfunction of modulus 1 {of $T_{\bar x}$ in $L^2_0(Z,\kappa)$ (identified 
with $L^2_0(X,\mu_{\bar x})$).}
\end{lemma}
\begin{proof}
	We will define $F$ on the set 
	$\overline{X}_{\infty}$ by considering a multifunction 
	$H$ defined below and taking $F$ as a measurable selector given by Theorem 
	\ref{th:selector}. Therefore, the reminder of the proof is devoted to 
	picking a proper multifunction and checking that the assumptions of Theorem 
	\ref{th:selector} are satisfied.
	
	Let $W\subset L^2_0(Z,\mathcal D, \kappa)$ be the subset of function of 
	integral 0 and modulus 1 (note that $W$ is a Polish space in 
	$L^2$-topology).
	Let $H$ be a multifunction which assigns to each element $T_{\bar 
	x}:=(T,\mu_{\bar x})$ the set $H(T_{\bar x})\subset W$ of its 
	eigenfunctions. Note that this is a closed set. 
	Indeed, assume that $f_n\to f\in L^2_0(Z,\mathcal D, \kappa)$, where $(f_n)_{n\in\n}$ is a sequence of 
	functions in $H(T_{\bar x})$. Then $f$ is also of modulus 1. Moreover, by 
	using the compactness of the circle and passing to a subsequence if 
	necessary, we also have that the sequence $(\la_n)_{n\in\n}$ of the 
	corresponding eigenvalues   converges to some number $\la \in\mathbb S^1$. 
	It is now easy to check that $f$ is an eigenfunction corresponding to the 
	eigenvalue $\la$.
	
	We now prove that $H$ is weakly measurable, so that we can apply Theorem~\ref{th:selector}.
	Let $A\subset W$ be open. Without loss of generality, we can assume that 
	$\overline X_{\infty}$ is a metric compact space. Consider the map 
	$\varphi_A:\overline{X}_{\infty}\times \mathbb S^1\times A\to 
	L^2(Z,\mathcal D, \kappa)$,  given by
	\[
	\varphi_A(\bar x, \la, f)=f\circ T_{\bar x}-\la f.
	\]
	By \eqref{ed}, $T_{\bar x}\in \operatorname{Aut}(Z,\mathcal D, \kappa)$. Hence, the right-hand side of the above formula is an element of   $L^2(Z,\mathcal D, \kappa)$. Let
	\[
	\pi_{\overline X}:\overline{X}_{\infty}\times \mathbb S^1\times A\to \overline{X},\ \pi_{\overline X}(\bar x,\lambda, f)=\bar x
	\]
	be the projection on the first coordinate.
	To prove the weak measurability of $H$, we need to show that the set
	\[
	\{\bar x\in \overline X_{\infty}:\, H(T_{\bar x})\cap A\neq\emptyset 
	\}=\pi_{\overline 
		X}(\varphi^{-1}_A({0})) 
	\]
	is measurable.
	Since the map $\bar x\mapsto T_{\bar x}$ is Borel measurable, then so is 
	$\varphi_A$. Thus $\pi_{\overline 
		X}(\varphi^{-1}_A({0})) $ is analytic, hence measurable wrt. $\mathcal 
		C$ being the $\sigma$-algebra of $P$-measurable sets. It remains to  
		use Theorem \ref{th:selector} for $\Omega=\overline X_{\infty}$ and $\mathcal C$.

%
	\end{proof}
\begin{remark}  The following result follows immediately from the fact 
that the class of automorphisms whose a.a.\ ergodic components have discrete 
spectra is a characteristic class as proved in \cite{Ka-Ku-Le-Ru}. Hence, for 
$T$, there exists the largest factor of 
$T$ belonging to this characteristic class, and this factor considered on each 
fiber $T_{\bar x}$ is the Kronecker 
factor of $T_{\bar x}$, see \cite{Ka-Ku-Le-Ru}, Section 2.3.1. However, we need 
a special form of this factor, hence we give the complete proof below.
\end{remark}

\begin{lemma}\label{cor:measfactor}
	Let $T$ satisfy the assumptions of Lemma \ref{prop:measfactor}. Then there 
	exists a non-trivial factor of $T$ whose ergodic decomposition 
	consists of ergodic rotations.
\end{lemma}
\begin{proof}
	Let $F$ be given by Lemma \ref{prop:measfactor}. Note that for every $\bar 
	x\in \overline{X}_{\infty}$ the corresponding eigenvalue $\varphi(\bar x)$ is equal 
	to $\frac{F(\bar x)\circ T_{\bar x}}{F(\bar x)}$, hence it depends 
	measurably on $\bar x$. Thus, in view of the fact that the map $\mathbb 
	S^1\ni\al\mapsto R_\al\in {\operatorname{Aut}(\mathbb S^1, Leb_{\mathbb S^1})}$ is 
	continuous, the 
	automorphism $S\in \operatorname{Aut}(\overline X\times \mathbb S^1,P\otimes Leb_{\mathbb 
	S^1})$,  
\begin{equation}\label{eq: factorform}
S(\bar x, r):=(\bar x, \varphi(\bar x)r)
\end{equation}
is the desired factor 
	and factorizing map $J:\overline X\times Z\to \overline X\times \mathbb 
	S^1$ is given by
	\begin{equation*}
	J(\bar x, z)=\left(\bar x,F(\bar x)(z)\right).
	\end{equation*}
	\end{proof}
{


\noindent
\begin{proof}[{\textit Proof of Theorem~\ref{t:main}.}]
Let $(X,\mathcal B,\mu, T)\in \mathcal M(Erg^{\perp})$ and let 
$\mu=\int_{\overline{X}}\mu_{\bar x}\,dP(\bar x)$  be its ergodic 
decomposition. In view of Corollary \ref{cor: Xn0mes}, we have $P(\overline X_1\cup \overline X_{\infty})=1$. We will show that $P(\overline X_1)=1$. Assume by contradiction, that $P(\overline X_\infty)>0$.

Consider first  the case when  $P\left(\mathcal{WM}\right)>0$, where 
$\mathcal{WM}:=\{\bar x\in \overline{X}_{\infty}: T_{\bar x}\text{ is weakly mixing}\}$ 
(it is measurable via Lemma~\ref{lem:wm}). {Then $T$ can be 
decomposed into a disjoint action of two automorphisms $\tilde T$ and 
$\tilde{T}^\ast$.  Here, $\tilde{T}$ stands for the restriction of $T$ to the 
union of fibers corresponding to $\mathcal{WM}$, i.e., it is measure 
$\tilde\mu$-preserving,} where the ergodic decomposition of $\tilde{T}$ is of 
the form
\[
\tilde\mu=\frac{1}{P(\mathcal{WM})}\int_{\mathcal{WM}}\mu_{\bar x}\, dP(\bar x)
\]
and the automorphism $\tilde T^*$  is the restriction of $T$ to the union of 
the fibers corresponding to $\overline X\setminus \mathcal{WM}$.
Then $(\tilde{T},\tilde{\mu})$ satisfies the assumption of Lemma \ref{prop: wmcase} and thus, it is not a multiplier of $Erg^{\perp}$, that is, there exists $S\in Erg^\perp$ and a joining $\tilde\eta\in J(\tilde T,S)$ such that $(\tilde T\times S,\tilde\eta)\notin Erg^\perp$. It is now enough to join $\tilde T^*$ and $S$ independently, to get a non-trivial joining $\eta\in J(T,S)$ such that $(T\times S,\eta)\notin Erg^\perp$, which is a contradiction with the fact that $T\in Erg^{\perp}$.

Thus, we can assume that for $P$-a.e.\ $\bar x\in\overline X_{\infty}$, the automorphism $T_{\bar x}$ has a non-trivial eigenvalue. Let $\mu_{\infty}=\frac{1}{P(\overline X_{\infty})}\int_{\overline X_{\infty}} \mu_{\bar x}\, dP(x)$.
Then, by Lemma \ref{cor:measfactor}, there exists a 
factor  $\hat T$ of $(T,\mu_{\infty})$ such that all ergodic components $\hat T_{\bar x}$ of $\hat T$ 
are rotations on the circle. Note that for a.e.\ pair of $(\bar x,\bar y)\in 
\overline{X}_{\infty}\times \overline{X}_{\infty}$, the associated rotations are disjoint in view 
of Theorem \ref{classification}. Then, $\hat T$ has exactly the form of Example~\ref{example1}, which is not in $\mathcal M(Erg^{\perp})$. Since, by Corollary \ref{cor:multclass}, $\mathcal M(Erg^{\perp})$ is a characteristic class, we get that $(T,\mu_{\infty})\notin \mathcal M(Erg^{\perp})$. Thus, again by joining the restriction of $T$ to the fibers in $\overline X_1$ independently, we get that $(T,\mu)\notin \mathcal{M}(Erg^\perp)$.

Hence $P(\overline X_1)=1$, which means that $T$ is an identity. This proves that $\mathcal{M}(Erg^\perp)\subset ID$. The opposite inclusion follows directly from Proposition \ref{idperp+}. This finishes the proof.
\end{proof}
	\section{$Erg^{\perp}$ is closed under taking products}\label{sec: products}
{As Theorem~\ref{t:main} shows, the class $Erg^{\perp}$ is not closed under 
taking joinings due to an interplay between fiber automorphisms over the 
``common'' part of ergodic components. It turns out that this phenomenon can 
not happen if the spaces of ergodic components are independent -- in this 
section we show the  class $Erg^\perp$ is actually closed under taking the 
Cartesian products.}

\begin{theorem}\label{closedonproducts}
	Assume that $(X,\mathcal B,\mu, T), (Y,\mathcal C,\nu, S)\in Erg^\perp$. Then
	\[
	(X\times Y,\mathcal B\otimes\mathcal C,\mu\otimes\nu, T\times S)\in Erg^\perp.
	\]
\end{theorem}
\begin{proof}
	Let $(Z,\mathcal D,\rho, R)$ be an arbitrary ergodic system. Recall first that all systems under consideration may be viewed as measure-preserving homeomorphisms of compact metric spaces. The above observation allows one to utilize Lemma~\ref{msrbl:val} later in the proof.
	
	On the space $X\times Y\times Z$ for every $A\in\{X,Y,Z,X\times Y,X\times 
	Z,Y\times Z\}$ denote by $\pi^{A}:X\times Y\times Z\to A$ the standard 
	projection on the corresponding coordinates. Consider any joining $\la\in 
	J((T\times S,\mu\otimes\nu),(R,\rho))$. We aim at showing that 
	$\la=\mu\otimes\nu\otimes\rho$.
	
{First, note that $(\pi^{X,Z})_*\la=\mu\otimes\rho$ and $(\pi^{Y,Z})_*\la=\nu\otimes\rho$ as $T,S\in Erg^\perp$. 
Let
	\[
	\mu=\int_{\overline{X}}\mu_{\bar x}\ dP(\bar x)
	\]
	be the ergodic decomposition of $\mu$ (with  $P=\mu|_{Inv(T)}$). Since 
	$Inv(T)$ is also a factor of $(T\times S\times R,\lambda)$ and 
	$(\pi^X)_\ast\lambda=\mu$,
	\begin{equation}\label{dis:la}
	\la=\int_{\overline{X}}\la_{\bar x}\ dP(\bar x)
	\end{equation}
and $(\pi^X)_*\la_{\bar x}=\mu_{\bar x}$ $P$-a.e., by the uniqueness of disintegrations.} 
	Moreover, note that
	\[
	\mu\otimes\rho=\int_{\overline{X}}\mu_{\bar x}\otimes\rho\ dP(\bar x).
	\]
	Since $(\pi^{X,Z})_*\la=\mu\otimes\rho$, again by using the uniqueness of disintegration, we get
	\begin{equation}\label{projXZ}
	(\pi^{X,Z})_*\la_{\bar x}=\mu_{\bar x}\otimes\rho\quad P\text{-a.e.}
	\end{equation}
	Analogously, since $(\pi^{X,Y})_*\la=\mu\otimes\nu$, we also have
	 	\begin{equation}\label{projXY}
	 (\pi^{X,Y})_*\la_{\bar x}=\mu_{\bar x}\otimes\nu                 \quad P\text{-a.e.}
	 \end{equation}
	 In particular, $(\pi^{Y})_*\la_{\bar x}=\nu$. Thus, by \eqref{projXZ}, we 
	 have
	 \[
	 \la_{\bar x}\in J((S,\nu),(T\times R,\mu_{\bar x}\otimes \rho))\quad P\text{-a.e.}
	 \]
	 What remains to show is that
	 \begin{equation}\label{mes:nerg}
	 P(\{\bar x:\ (T\times R,\mu_{\bar x}\otimes \rho)\ \text{is not ergodic}\})=0.
	 \end{equation}
	 Indeed, if \eqref{mes:nerg} holds then since $S\in Erg^\perp$ we obtain that
	 \[
	 \la_{\bar x}=\mu_{\bar x}\otimes \nu\otimes\rho\quad\text{ $P$-a.e.},
	 \]
	 which together with \eqref{dis:la} yields $\la=\mu\otimes\nu\otimes\rho$.
	
	 To show \eqref{mes:nerg}, recall first that $R$ can have at most countably 
	 many eigenvalues. Moreover, the Cartesian product of two ergodic systems 
	 is not ergodic if and only if they have a non-trivial common eigenvalue. 
	 It is thus enough to show that for any fixed $\al\in\mathbb 
	 T\setminus\{0\}$ we have
	 \begin{equation}\label{mes:nerg2}
	 P(\{\bar x:\ \al\ \text{is an eigenvalue of}\ (T,\mu_{\bar x})\})=0.
	 \end{equation}
%
	Assume that \eqref{mes:nerg2} does not hold, that is,
	 \begin{equation}\label{mes:pos}
	 P(\{\bar x:\ \al\ \text{is an eigenvalue of}\ (T,\mu_{\bar x})\})>0.
	 \end{equation}
	 Then by Lemma \ref{lm: nonergeigenv}, $\alpha$ is an eigenvalue of $T$. 
	 This is a contradiction with $\eqref{commonev}$. Hence, we proved 
	 \eqref{mes:nerg2} which in turn completes the proof of the theorem.
	\end{proof}
\noindent By induction, we obtain from the above result the following corollary.
\begin{corollary}
	If  ${(X_i,\mathcal B_i,\mu_i, T_i)}_{i\in\n}$ is a sequence of elements of $Erg^{\perp}$ then
	\[
	\left(\prod_{i=0}^{\infty}X_i,\bigotimes_{i=0}^{\infty}\mathcal B_i,\bigotimes_{i=0}^{\infty}\mathcal \mu_i, \prod_{i=0}^{\infty}T_i\right)\ \in\ Erg^{\perp}.
	\]
	\end{corollary}

\section{An automorphism whose self-joinings are disjoint with ergodic automorphisms}\label{sec: selfj}
We are going to construct an automorphism whose self-joinings are all elements of $Erg^{\perp}$. To do that, we will rely on the PID property. We need the following result.
\begin{theorem}[Ryzhikov 1994]\label{th:rhyz}
	Assume that $(X,\mathcal B,\mu, T)$ has the PID property and $(Y_i,\mathcal B_i,\nu_i,S_i)$, $i=1,2$, are measure preserving automorphisms. If  $\la\in J(T,S_1,S_2)$ projects as product measures on each pair of coordinates then
	\[
	\la=\mu\otimes\nu_1\otimes\nu_2.
	\]
	\end{theorem}
In the proof of Theorem \ref{thm: self-join}, we will use several times the following corollary of the above result.
\begin{corollary}\label{cor: multiPID}
Let $n\ge 2$ and let  $(X_i,\mathcal B_i,\mu_i,T_i)$, $i=1,\ldots,n$, be measure preserving automorphisms satisfying the PID property. Let $(Y,\mathcal C,\nu,S)$ be a measure preserving automorphism. If  $\la\in J(S,T_1,\ldots,T_n)$ projects as product measures on each pair of coordinates then
\[
\la=\nu\otimes\mu_1\otimes\ldots\otimes\mu_n.
\]
\end{corollary}
\begin{proof}
	We use induction. For $n=2$, the result follows directly from Theorem \ref{th:rhyz}. Assume that it is true for some $n\ge 2$, we will prove that it holds for $n+1$. Let $\la\in J(S,T_1,\ldots,T_{n+1})$ and let $\tilde \la$ be the projection of $\la$ on $Y\times X_1\times\ldots\times X_n$ and let $\bar\la$ be the projection of $\la$ on $X_1\times\ldots\times X_{n+1}$. By the induction assumption, we get that
	\[
	\tilde\la=\nu\otimes\mu_1\otimes\ldots\otimes\mu_{n}\ \text{ and }\ \bar\la=\mu_1\otimes\ldots\otimes\mu_{n+1}.
	\]
	The result now follows from Theorem \ref{th:rhyz} by taking $T:=T_{n+1}$, $S_1:=S$ and $S_2:=T_1\times\ldots\times T_n$.

	\end{proof}
We are now ready to present and prove the main result of this section.
\begin{theorem}\label{thm: self-join}
	Let $T:X\to X$ be a $\mu$-preserving homeomorphism of a compact metric space $X$. Let
	\[
	\mu=\int_{\overline{X}}\mu_{\bar x}\ dP(\bar x)
	\]
	be the ergodic decomposition and let $T_{\bar x}$ denote the action of $T$ on the fiber corresponding to $\bar x\in \overline{X}$. Assume that $P$ is a continuous probability measure. Assume moreover that:
	\begin{itemize}
		\item for every $\bar x\in \overline{X}$ there exists a set of $\bar y\in \overline{X}$, whose complement is countable, such that for every $\bar y$ from this set, we have $T_{\bar x}\perp T_{\bar y}$;
		\item for every $\bar x,\bar y\in \overline{X}$ if $T_{\bar x}\notperp T_{\bar y}$  then $T_{\bar x}$ and $T_{\bar y}$ are isomorphic;
		\item for every $\bar x\in \overline{X}$ the automorphism $(T_{\bar x},\mu_{\bar x})$ has the MSJ property (in particular, it has the PID property).
	\end{itemize}
Then $(T^{\times \infty},\eta)\in Erg^{\perp}$ for every 
$\eta\in J_\infty(T,\mu)$.
\end{theorem}
\begin{proof}
	It is enough to prove that for every $n\ge 1$ and every  $\eta\in J_n(T,\mu)$, $(T^{\times n},\eta)\in Erg^\perp$. 	Note that since $P$ is continuous, for $P\otimes P$-a.e. $(\bar x, \bar y)\in \overline X\times \overline X$, we have $T_{\bar x}\perp T_{\bar y}$. Then, by Theorem \ref{classification}, $T\in Erg^\perp$.
	
	Let $2\le n<\infty$, fix $\eta\in J_n(T,\mu)$ and let $(R,\rho)\in Erg$ be 
	arbitrary. 
Let $\psi\in 
J((T^{\times n},\eta),(R,\rho))$. We want to show that $\psi=\eta\otimes 
\rho$. To distinguish between the $n$ copies of $X$ and $\overline{X}$ we 
denote by $X_k$ and $\overline X_k$ the domain and the space of ergodic 
components of $T$ on the $k$-th coordinate for every $k=1,\ldots n$. We assign 
the  coordinate $(n+1)$ to the automorphism $R$. Finally, we set
$\mathbf X=X_1\times\ldots\times X_n$ and by 
$\pi^{k_1,\ldots,k_m}:\mathbf X\to X_{k_1}\times\ldots \times X_{k_m}$ we denote
the 
projection on $k_1,\ldots,k_m$ coordinates.

Consider the ergodic decomposition of $(T^{\times n},\eta)$:
\[
\eta=\int_{\overline{\mathbf X}}\eta_{\bar {\mathbf x}}\ dQ(\bar {\mathbf 
x}).
\]
Thus we also have the following disintegration of $\psi$:
\[
\psi=\int_{\overline{\mathbf X}}\psi_{\bar {\mathbf x}}\ dQ(\bar {\mathbf 
	x}).
\]
By uniqueness of ergodic decomposition we have that $(\pi^{\mathbf X})_*\psi_{\bar{\mathbf x}}=\eta_{\bar{\mathbf x}}$ and by ergodicity of $R$ we get $(\pi^{n+1})_*\psi_{\bar{\mathbf x}}=\rho$. Thus  $\psi_{\bar{\mathbf x}}\in 
J((T^{\times n},\eta_{\bar{\mathbf x}}), (R,\rho))$ for $Q$-a.e. $\bar{\mathbf 
x}$. To 
show that $\psi=\eta \otimes \rho$, we just have to show that 
\begin{equation}\label{eq: prodoncoord}
\psi_{\bar {\mathbf x}}=\eta_{\bar{\mathbf x}}\otimes\rho\quad\text{for 
$Q$-a.e. }\bar{\mathbf x}\in \overline{\mathbf X}.
\end{equation}

Note that by uniqueness of ergodic decomposition, for $Q$-a.e. 
$\bar{\mathbf x}\in \overline{\mathbf X}$ and for every $k=1,\ldots,n$, we have 
\[
(\pi^k)_*\psi_{\bar {\mathbf x}}=(\pi^k)_*\eta_{\bar {\mathbf x}}=\mu_{\bar 
x_k}\quad\text{for some }\bar x_k(\bar{\mathbf x})=\bar x_k\in \overline{X}_k.
\] 
Moreover, all marginals of $Q$ are equal to $P$.
Therefore, since $P$ is continuous, we get
that for every $\bar x\in \overline X$ we have 
\begin{equation}\label{eq: noprevalentcomponent}
Q\left\{\bar{\mathbf x}\in \overline{\mathbf X}:\ 
\mu_{\bar x_k}=\mu_{\bar x} \text{ for some }k=1,\ldots,n  \right\}=0.
\end{equation}
In particular, 
since there are only up to countably many ergodic components of $T$ isomorphic to $\mu_{\bar x}$, 
we get by \eqref{eq: noprevalentcomponent} that for every $\bar 
x\in\overline{X}$
\begin{equation}\label{eq: noprevalentclass}
	Q\left(\left\{\bar{\mathbf x}\in \overline{\mathbf X}:\ (T,\mu_{\bar x_k})\text{ 
	is isomorphic to }(T,\mu_{\bar x})  \text{ for some  
			}k=1,\ldots,n  \right\}\right)=0.
\end{equation}

Now, to show \eqref{eq: prodoncoord}, we consider cases which depend on the 
form of $\eta_{\bar{\mathbf x}}$. More precisely, we consider the number of 
coordinates on which the projection yields isomorphic maps (recall that by 
assumption, the ergodic components are either isomorphic or disjoint).

\vspace{5mm}
\noindent\textbf{Case 1: No isomorphic components.} Let $\mathcal D_0\subset \overline{\mathbf X}$ be the set of elements 
satisfying the following
\begin{equation}\label{eq: pairwisedisjoint}
(T,\mu_{\bar x_1}),\ldots,(T,\mu_{\bar x_n}) 
\text{ are pairwise disjoint}.
\end{equation}
Then, by Corollary \ref{cor: multiPID} and mutual disjointness assumption, for 
every $\bar{\mathbf x}\in \mathcal D_0$ we get 
\begin{equation}\label{eq: product}
\eta_{\bar {\mathbf x}}=\mu_{\bar x_1}\otimes \ldots\otimes\mu_{\bar x_n}.
\end{equation}
Note also that for $Q$-a.e. $\bar{\mathbf x}\in \mathcal D_0$, the automorphism 
$(R,\rho)$ is disjoint 
with $(T,\mu_{\bar x_k})$ for every $k=1,\ldots, n$. Indeed, otherwise for some $1\le 
k_0\le n$ we would have that \footnote{The measurability of sets considered below follows from Section 3.4 in \cite{Au-Go-Le-Ru}.}
\[\begin{split}
	Q&\left(\left\{\bar{\mathbf x}\in \overline{\mathbf X};\ (T,\mu_{\bar 
	x_{k_0}})\text{ and }(R,\rho) \text{ are not disjoint } \right\}\right)\\&=
	P\left(\left\{\bar{x}\in \overline{X};\ (T,\mu_{\bar x})\text{ and 
	}(R,\rho) \text{ are not disjoint } \right\}\right)>0.
\end{split}
\]
By Proposition \ref{disjuptocount} and the assumptions of the theorem, the set 
considered above can be at most countable. This is a contradiction with 
the fact that $P$ is continuous.

Thus, for $Q$-a.e. $\bar{\mathbf x}\in \mathcal D_0$, the automorphism $(R,\rho)$ is 
disjoint with $(T,\mu_{\bar x_k})$ for every $k=1,\ldots, n$. In particular, by 
\eqref{eq: product}, for every such $\bar{\mathbf x}$, the measure 
$\psi_{\bar{\mathbf x}}$ projects on each pair of coordinates as a product 
measure. By Corollary \ref{cor: multiPID}, we get that 
\[
\psi_{\bar{\mathbf x}}=\eta_{\bar{\mathbf x}}\otimes \rho,
\]
for $Q$-a.e. $\bar{\mathbf x}\in \mathcal D_0$. This finishes the proof of Case 
1.  

\vspace{5mm}
\noindent\textbf{Case 2: There exist isomorphic components.} Consider now the 
elements $\bar{\mathbf x}\in\overline{\mathbf X}$ such that, up to a permutation of coordinates, there exist $m<n$ 
and indices $0=\ell_0< \ell_1<\ldots<\ell_{m-1}<\ell_m=n$ such that 
$(T,\mu_{\bar x_{\ell_{k-1}+1}}),\ldots, (T,\mu_{\bar x_{\ell_k}})$ are isomorphic for every 
$k=1,\ldots, m$ and $m$ is minimal in this representation. For every $m=1,\ldots,n-1$, denote by $\mathcal D_m\subset 
\overline{\mathbf X}$, the set of elements for which there are exactly $m$ 
groups of indices in the above decomposition. We will show that this case 
reduces to the Case 1.

Fix $m=1,\ldots,n-1$ and let $\bar{\mathbf x}\in \mathcal D_m$. Let 
$\ell_0,\ldots,\ell_m$ be as above. Since $\eta_{\bar{\mathbf x}}$ is ergodic, 
for every $k=1,\ldots,m$, the projection 
$(\pi^{\ell_{k-1}+1,\ldots,\ell_{k}})_*\eta_{\bar{\mathbf x}}$ is an ergodic 
measure for the map $T^{\times \ell_{k}-\ell_{k-1}}$ (up to an isomorphism it is a self-joining of $(T,\mu_{\bar x_{\ell_k}})$). Moreover, recalling that $(T,\mu_{\bar x})$ has the MSJ property for every $\bar 
x\in \overline X$, we obtain that
	\[
(\pi^{\ell_{k-1}+1,\ldots,\ell_{k}})_*\eta_{\bar{\mathbf x}} \text{ is a product 
of $r_k\le \ell_{k}-\ell_{k-1}$ off-diagonal self-joinings.}
	\]
Thus, the measure $\eta_{\bar{\mathbf x}}$ is actually a joining of 
$r_1+\ldots+r_m$ automorphisms, such that any two of them are either disjoint or isomorphic and in the latter case $\eta_{\bar{\mathbf x}}$ projects on the corresponding coordinates as the product measure.  Hence, by using Corollary \ref{cor: multiPID}, we get that $\eta_{\bar{\mathbf x}}$ 
is a product joining of $r_1+\ldots+r_m$ ergodic components of $T$. In other words, we reduced the problem, where \eqref{eq: product} is satisfied for a smaller number of indices.  The 
remainder of the proof of Case 2. is analogous to the proof of Case 1. by taking $n:=r_1+\ldots+r_m$.

\vspace{5mm}

\end{proof}
Recalling the definition of a characteristic class and  $\mathcal F(T)$ from Subsection \ref{subsec:charclass}, we get the following result.
\begin{corollary}\label{cor: noergcharclass}
	If $T$  satisfies the assumptions of Theorem \ref{thm: 
	self-join}, then $\mathcal F(T)$ is a characteristic class {of elements from $Erg^\perp$.}
\end{corollary}
\begin{proof}
	By Lemma \ref{lem: gencharclass}, Theorem \ref{thm: self-join} and the fact that 
	any factor of an element of $Erg^\perp$ is also a member of this class.
\end{proof}

We now construct an example to show that the set of automorphisms satisfying the assumptions of Theorem \ref{thm: self-join} is non-empty. We build it by using the classical cutting and stacking construction of systems of rank 1. 
We have the following fact, which is a special case of the results by Gao and Hill \cite{Gao-Hill} and Danilenko \cite{Dan}.
\begin{proposition}\label{rank1disj}
	Let $\bar a=\{a_n\}_{n\in\n}\in \{0,1\}^{\mathbb N}$ be the binary expansion of $a\in [0,1]$. Let $T_{a}:[0,1]\to[0,1]$ be a rank 1 automorphism defined in the following way, via cutting and stacking procedure:
	\begin{itemize}
		\item on each step the cutting parameter is equal to 3;
		\item on step $n$ we add a spacer over the first tower if $a_n=0$ and over the second if $a_n=1$.
			\end{itemize}
Assume that $|a-b|\neq \frac{k}{2^l}$ for any $k,l\in\n$. Then $T_a $ and $T_b$ are disjoint. Otherwise, they are isomorphic.
\end{proposition}
By \cite{dJRS}, the systems considered in the above proposition are ergodic and satisfy the MSJ property.
Note that the sequence heights of towers in the cutting and stacking 
construction described in the above theorem is universal for all $a\in[0,1]$. 
We leave the proof of the following easy lemma to the reader.
\begin{lemma}\label{rank1cont}
	The map $[0,1]\ni a\mapsto T_a\in \operatorname{Aut}(X,\nu)$ is well defined and continuous in $[0,1]\setminus \mathbb Q$.
	\end{lemma}
Let us consider $S\in \operatorname{Aut}([0,1]\times X, Leb_{[0,1]}\otimes \nu)$ defined in the following way
\[
S(a,x):=(a,T_a x).
\]
Then in view of Proposition \ref{rank1disj} and Lemma \ref{rank1cont}, by previous remarks, $S$ satisfies the assumptions of Theorem \ref{thm: self-join}. In particular,  we have $S\in Erg^{\perp}$.

{
	\section{A remark on multipliers}
We will now show some connections between the multipliers of a class $\mathcal{A}^\perp$ and  the characteristic classes included in $\mathcal{A}^\perp$. We recall that $\mathcal M(\mathcal A^{\perp})$ is always a characteristic class (see  Corollary \ref{cor:multclass}).
\begin{proposition}\label{prop: product}
For any $T$ and $R$, we have that
$T\in\mathcal{M}(\{R\}^\perp)$ provided that  for each $\lambda\in 
J_2(T)$, we have $(T\times T,\lambda)\perp R\times R$, where $R\times R$ is considered with product measure.
\end{proposition}
The proof of the proposition follows by word for word repetition (ignoring the ergodicity of $T$) of the proof of Prop. 5.1 in \cite{Le-Pa}.}
{
	Notice moreover that this proposition applies (non-
	trivially) only if $R\in WM$ (because we consider all self-joinings of $T$, hence also
	non-ergodic; on the other hand, $R \times R$ which is considered with product measure is not ergodic whenever $R$ is not weakly mixing; finally use that two
	non-trivial non-ergodic automorphisms are not disjoint).}
\begin{corollary}\label{mult-cart} Assume that $\mathcal{A}$ is a class 
closed under taking Cartesian squares. Then 
$T\in\mathcal{M}(\mathcal{A}^\perp)$ if and only if  for each $\lambda\in 
J_2(T)$, we have $(T\times T,\lambda)\perp R$ for each $R\in
\mathcal{A}$.
\end{corollary}
Note that the necessity in the corollary is obvious and the other direction follows from Proposition \ref{prop: product}. Similarly to Proposition \ref{prop: product}, the above corollary only applies when $\mathcal A\subset WM$.
Corollary \ref{mult-cart} implies also the following result.
{
\begin{corollary}\label{cart2} Assume that $\mathcal{A}$ is a class closed under taking Cartesian squares. Then $\mathcal{M}(\mathcal{A}^\perp)$ is the largest characteristic class
included in $\mathcal{A}^\perp$. In particular, the result holds when $\mathcal{A}=WM$.
\end{corollary}
Note that $ID$ satisfies the assumptions of Corollary \ref{cart2}.
It would be interesting to know, whether the assertion of the Corollary \ref{cart2} holds for any class $\mathcal A$.
Notice that for $\mathcal{A}= Erg$  the assumption on the Cartesian squares is not satisfied as the Cartesian square of any (non-trivial) rotation is not ergodic. More than that, in  $Erg^\perp$ the class of multipliers is $ID$, and it is the smallest characteristic class  included in $Erg^\perp$ (the existence of larger than $ID$ characteristic classes follows from Corollary \ref{cor: noergcharclass}), so the assertion of Corollary \ref{cart2} fails for $Erg^\perp$.}

{
A general question arises, whether given a characteristic class $\mathcal{C}\subset \mathcal A^{\perp}$, there exists a maximal characteristic class $\mathcal{C'}\subset \mathcal A^{\perp}$ such that $\mathcal{C}\subset\mathcal{C'}$.
To see that this holds, we apply Kuratowski-Zorn Lemma. Consider a chain $(\mathcal{C}_i)_{i\in\mathbb{I}}$ of characteristic classes such that $\mathcal{C}\subset\mathcal C_i\subset\mathcal A^{\perp}$. Let $\tilde{\mathcal C}$ be the smallest characteristic class containing $\bigcup_{i\in \mathbb{I}} \mathcal{C}_i$. We need to show that $\tilde{\mathcal C}\subset \mathcal A^{\perp}$.
 It is enough to check that if $T_n\in \bigcup_{i\in \mathbb{I}} \mathcal{C}_i$, $n=1,2,\ldots$, then  for every $\eta \in J(T_1,T_2,\ldots)$ we have $(T_1\times T_2 \times \ldots, \eta)\in \mathcal{A}^\perp$.}\\
{
Let $T_k \in \mathcal{C}_{i_k}$, $k=1,\ldots, n$, then, since $(\mathcal C_i)_{i\in\mathbb I}$ is a chain, wlog. we can assume that $T_1, T_2,\ldots, T_n \in \mathcal{C}_{i_n}$. Since $\mathcal{C}_{i_n}\subset \mathcal A^{\perp}$ is a characteristic class, we receive that all joinings of $T_1, T_2,\ldots, T_n$ are in $\mathcal{C}_{i_n}\subset \mathcal A^{\perp}$. To conclude, it remains to use the fact that every infinite joining is an inverse limit of finite joinings.}
{
\begin{corollary}
	For every $T$ satisfying the assumptions of Theorem \ref{thm: self-join}, there exists maximal characteristic class $\mathcal{C}\ni T$, such that $\mathcal{M}(Erg^\perp)\subset \mathcal{C}\subsetneq Erg^{\perp}$.
\end{corollary}
We have not been able to describe the maximal characteristic classes in $Erg^{\perp}$, it is even not clear whether there is only one.

\vspace{5mm}

\noindent \textbf{Acknowledgements.} We thank Mariusz Lemańczyk for numerous comments and discussions. The first author was supported by NCN Grant 2022/45/B/ST1/00179. This article is also a part of the project IZES-ANR-22-CE40-0011.

	\noindent
Faculty of Mathematics and Computer Science, Nicolaus Copernicus University,\\
Chopina 12/18, 87-100 Toruń, Poland\\
zimowy@mat.umk.pl, gorska@mat.umk.pl\\
\smallskip

\noindent
Univ Rouen Normandie, CNRS, Normandie Univ, LMRS UMR 6085, F-76000\\
Rouen, France\\
thierry.de-la-rue@univ-rouen.fr


\begin{thebibliography}{99}
\bibitem{AIM} American Institute of Mathematics, workshop Sarnak’s Conjecture, December 2018,
http://aimpl.org/sarnakconjecture/3/
\bibitem{Au-Go-Le-Ru} T. Austin, M. G\'orska, M.\ Lema\'nczyk and T. de la Rue, in preparation.
\bibitem{Dan} A. Danilenko, \emph{Rank-one actions, their (C; F)-models and constructions with bounded
parameters,} J. Anal. Math. 139 (2019), 697–749.
\bibitem{Da-Ry}A.I. Danilenko, V.V. Ryzhikov, {\em On self-similarities of ergodic flows}, Proc. Lond. Math. Soc. (3) 104 (2012), 431–454.
\bibitem{Fr-Ho} N. Frantzikinakis, B. Host, {\em The logarithmic Sarnak conjecture for ergodic weights},
Annals Math. (2), 187(3) (2018), 869–931.
\bibitem{Fu} H. Furstenberg, {\em Disjointness in ergodic theory, minimal sets and Diophantine
approximation}, Math. Systems Theory 1 (1967), 1–49.
\bibitem{Fu-We} H. Furstenberg, B. Weiss, {\em The finite multipliers of infinite ergodic transformations}, in: Lecture Notes in Math. 668, Springer, 1978, 127–132.
\bibitem{Gao-Hill} Su  Gao, A. Hill, \emph{Disjointness between bounded rank-one transformations,} Colloq. Math. 164 (2021), no. 1, 91–121.
\bibitem{Gl1} E. Glasner, {\em On the multipliers of $W^\perp$}, Ergodic Theory Dynam. Sys. 14 (1994), 129-140.
\bibitem{Gl2} E. Glasner, {\em Ergodic Theory via Joinings}, Mathematical Surveys and Monographs 101, AMS,
Providence, RI, 2003.
\bibitem{Gl-We} E. Glasner, B. Weiss, {\em Processes disjoint from weak mixing}, Trans. Amer. Math. Soc. 316 (1989), 689-703.
\bibitem{Ha-Pa} F. Hahn and W. Parry, {\em Some characteristic properties of dynamical systems with
quasi-discrete spectrum}, Math. Systems Theory 2 (1968), 179–198
\bibitem{Ha}P.R. Halmos, 
\emph{In general a measure preserving transformation is mixing,} 
Ann. of Math. (2) 45 (1944), 786–792. 
\bibitem{dJRS} A. del Junco, M.  Rahe, L. Swanson,\emph{ Chacon's automorphism has minimal self-joinings,} J. Analyse Math. 37 (1980), 276–284.
\bibitem{Ju-Ru}A.\ del Junco, D.\ Rudolph, \emph{On ergodic actions whose self-joinings are graphs}, Ergodic Theory Dynam. Systems 7 (1987), no. 4, 531–557.

\bibitem{Ka-Ku-Le-Ru} A.\ Kanigowski, J.\ Ku\l aga-Przymus, M.\ Lema\'nczyk, T.\ de la Rue, {\em On arithmetic functions orthogonal to deterministic sequences},
Advances Math. {428} (2023), 68 pp.

\bibitem{Kechris} A. S. Kechris\ \emph{Global Aspects of Ergodic Group Actions}, Mathematical Surveys and Monographs
Volume: 160; 2010; 237;

\bibitem{KuRyN} K. Kuratowski;  C. Ryll-Nardzewski\ \emph{A general theorem on 
selectors}. Bull. Acad. Polon. Sci. Sér. Sci. Math. Astronom. Phys. 13 (1965): 
397–403.
\bibitem{Le-Pa} M. Lema\'nczyk, F. Parreau, \emph{ Rokhlin extensions and 
lifting 
disjointness}, Ergodic Theory Dynam. Systems 23 (2003), 1525-1550.
\bibitem{Le-Pa-Th} M. Lema\'nczyk, F. Parreau and J.-P. Thouvenot, \emph{ 
Gaussian automorphisms whose ergodic self-joinings are
Gaussian}, Fundamenta Math. 164 (2000), 253–293.
\bibitem{rokhlin} V. A. Rokhlin, \emph{On the fundamental ideas of measure theory}, Mat. Sb. (N.S.), 25(67):1 (1949), 107–150
\bibitem{Ru} D. Rudolph, \emph{An example of a measure preserving map with minimal self-joinings, and applications}, J. Analyse Math. 35 (1979), 97–122.
\bibitem{Si} Ya. G. Sinai, {\em The structure and properties of invariant measurable partitions}, Dokl.
Akad. Nauk SSSR 141 (1961), 1038–1041.
 \end{thebibliography}
\end{document}